\documentclass[amstex,12pt, amssymb]{article}

\usepackage{mathtext}
\usepackage[cp1251]{inputenc}
\usepackage[T2A]{fontenc}
\usepackage[dvips]{graphicx}
\usepackage{amsmath}
\usepackage{amssymb}
\usepackage{amsxtra}
\usepackage{latexsym}
\usepackage{ifthen}

\newcounter{lemma}[section]

\newcounter{corollary}[section]

\newcounter{remark}[section]

\newcounter{theorem}[section]

\newcounter{proposition}[section]

\newcounter{example}

\numberwithin{equation}{section}

\begin{document}

\markboth{\centerline{E.~SEVOST'YANOV,
O.~DOVHOPIATYI}}{\centerline{ON COMPACT CLASSES...}}

\def\cc{\setcounter{equation}{0}
\setcounter{figure}{0}\setcounter{table}{0}}

\overfullrule=0pt


\author{O.~DOVHOPIATYI, E.~SEVOST'YANOV}

\title{
{\bf ON COMPACT CLASSES OF SOLUTIONS OF DIRICHLET PROBLEM IN SIMPLY
CONNECTED DOMAINS}}

\date{\today}
\maketitle

\begin{abstract}
The article is devoted to questions concerning the problems of
compactness of solutions of the Dirichlet problem for the Beltrami
equation in some simply connected domain. In terms of prime ends, we
have proved results of a detailed form for the case when the maximal
dilations of these solutions satisfy certain integral constraints.
In addition, in this article we have proved theorems on the local
and global behavior of plane and spatial mappings with direct and
inverse modulus conditions.
\end{abstract}

\section{Introduction}

In our recent joint publication~\cite{SD}, we proved the compactness
theorem of the classes of solutions of the Dirichlet problem for the
Beltrami equation in a simply connected Jordan domain whose
characteristics satisfy the constraints of the integral type. In
this article we are talking about solutions defined in an arbitrary
simply connected domain. Since such domains do not have to be
Jordanian, this is a slight relaxation of conditions compared
to~\cite{SD}. Note the existence of solutions of the Dirichlet
problem under certain assumptions (see, e.g., ~\cite{P}). Note that
such solutions, generally speaking, do not have a homeomorphic
extension to the boundary of the domain in the usual sense, since
simply connected domains do not have to be locally connected on
their boundary. However, such a continuation usually takes place in
terms of the so-called prime ends.

\medskip
Let $D$ be a domain in ${\Bbb C}.$ In what follows, a mapping
$f:D\rightarrow{\Bbb C}$ is assumed to be {\it sense-preserving,}
moreover, we assume that $f$ has partial derivatives almost
everywhere. Put $f_{\overline{z}} = \left(f_x + if_y\right)/2$ and
$f_z = \left(f_x - if_y\right)/2.$ The {\it complex dilatation} of
$f$ at $z\in D$ is defined as follows:
$\mu(z)=\mu_f(z)=f_{\overline{z}}/f_z$ for $f_z\ne 0$ and $\mu(z)=0$
otherwise. The {\it maximal dilatation} of $f$ at $z$ is the
following function:
\begin{equation}\label{eq1}
K_{\mu}(z)=K_{\mu_f}(z)=\quad\frac{1+|\mu (z)|}{1-|\mu\,(z)|}\,.
\end{equation}
Note that the Jacobian of $f$ at $z\in D$ may be calculated
according to the relation
$$J(z,
f)=|f_z|^2-|f_{\overline{z}}|^2\,.$$ Since we assume that the map
$f$ is sense preserving, the Jacobian of this map is positive at all
points of its differentiability. Let ${\Bbb D}=\{z\in {\Bbb C}:
|z|<1\},$ and let $\mu:D\rightarrow {\Bbb D}$ be a Lebesgue
measurable function. Without reference to some mapping $f,$ we
define the {\it maximal dilatation} corresponding to its complex
dilatation~$\mu$ by~(\ref{eq1}).
It is easy to see that
$$K_{\mu_f}(z)=\frac{|f_z|+|f_{\overline{z}}|}{|f_z|-|f_{\overline{z}}|}$$
whenever partial derivatives of $f$ exist at $z\in D$ and, in
addition, $J(z, f)\ne 0.$

Let $D$ be a domain in ${\Bbb R}^n,$ $n\geqslant 2.$ Recall some
definitions (see, for example,~\cite{KR$_1$}, \cite{KR$_2$},
\cite{IS} or \cite{SSI}). Let $\omega$ be an open set in ${\Bbb
R}^k$, $k=1,\ldots,n-1$. A continuous mapping
$\sigma\colon\omega\rightarrow{\Bbb R}^n$ is called a {\it
$k$-dimensional surface} in ${\Bbb R}^n$. A {\it surface} is an
arbitrary $(n-1)$-dimensional surface $\sigma$ in ${\Bbb R}^n.$ A
surface $\sigma$ is called {\it a Jordan surface}, if
$\sigma(x)\ne\sigma(y)$ for $x\ne y$. In the following, we will use
$\sigma$ instead of $\sigma(\omega)\subset {\Bbb R}^n,$
$\overline{\sigma}$ instead of $\overline{\sigma(\omega)}$ and
$\partial\sigma$ instead of
$\overline{\sigma(\omega)}\setminus\sigma(\omega).$ A Jordan surface
$\sigma\colon\omega\rightarrow D$ is called a {\it cut} of $D$, if
$\sigma$ separates $D,$ that is $D\setminus \sigma$ has more than
one component, $\partial\sigma\cap D=\varnothing$ and
$\partial\sigma\cap\partial D\ne\varnothing$.

A sequence of cuts $\sigma_1,\sigma_2,\ldots,\sigma_m,\ldots$ in $D$
is called {\it a chain}, if:

(i) the set $\sigma_{m+1}$ is contained in exactly one component
$d_m$ of the set $D\setminus \sigma_m,$ wherein $\sigma_{m-1}\subset
D\setminus (\sigma_m\cup d_m)$; (ii)
$\bigcap\limits_{m=1}^{\infty}\,d_m=\varnothing.$

Two chains of cuts  $\{\sigma_m\}$ and $\{\sigma_k^{\,\prime}\}$ are
called {\it equivalent}, if for each $m=1,2,\ldots$ the domain $d_m$
contains all the domains $d_k^{\,\prime},$ except for a finite
number, and for each $k=1,2,\ldots$ the domain $d_k^{\,\prime}$ also
contains all domains $d_m,$ except for a finite number.

The {\it end} of the domain $D$ is the class of equivalent chains of
cuts in $D$. Let $K$ be the end of $D$ in ${\Bbb R}^n$, then the set
$I(K)=\bigcap\limits_{m=1}\limits^{\infty}\overline{d_m}$ is called
{\it the impression of the end} $K$. Throughout the paper,
$\Gamma(E, F, D)$ denotes the family of all paths $\gamma\colon[a,
b]\rightarrow \overline{{\Bbb R}^n}$ such that $\gamma(a)\in E,$
$\gamma(b)\in F$ and $\gamma(t)\in D$ for every $t\in[a, b].$ In
what follows, $M$ denotes the modulus of a family of paths, and the
element $dm(x)$ corresponds to the Lebesgue measure in ${\Bbb R}^n,$
$n\geqslant 2,$ see~\cite{Va}. Following~\cite{Na$_2$}, we say that
the end $K$ is {\it a prime end}, if $K$ contains a chain of cuts
$\{\sigma_m\}$ such that
$\lim\limits_{m\rightarrow\infty}M(\Gamma(C, \sigma_m, D))=0$ for
some continuum $C$ in $D.$ In the following, the following notation
is used: the set of prime ends corresponding to the domain $D,$ is
denoted by $E_D,$ and the completion of the domain $D$ by its prime
ends is denoted $\overline{D}_P.$

Consider the following definition, which goes back to
N\"akki~\cite{Na$_2$}, see also~\cite{KR$_1$}--\cite{KR$_2$}. We say
that the boundary of the domain $D$ in ${\Bbb R}^n$ is {\it locally
quasiconformal}, if each point $x_0\in\partial D$ has a neighborhood
$U$ in ${\Bbb R}^n$, which can be mapped by a quasiconformal mapping
$\varphi$ onto the unit ball ${\Bbb B}^n\subset{\Bbb R}^n$ so that
$\varphi(\partial D\cap U)$ is the intersection of ${\Bbb B}^n$ with
the coordinate hyperplane.

\medskip
For the sets $A, B\subset{\Bbb R}^n$ we set, as usual,
$${\rm diam}\,A=\sup\limits_{x, y\in A}|x-y|\,,\quad {\rm dist}\,(A, B)=\inf\limits_{x\in A,
y\in B}|x-y|\,.$$
Sometimes we also write $d(A)$ instead of ${\rm diam}\,A$ and $d(A,
B)$ instead of ${\rm dist\,}(A, B),$ if no misunderstanding is
possible. The sequence of cuts $\sigma_m,$ $m=1,2,\ldots ,$ is
called {\it regular,} if
$\overline{\sigma_m}\cap\overline{\sigma_{m+1}}=\varnothing$ for
$m\in {\Bbb N}$ and, in addition, $d(\sigma_{m})\rightarrow 0$ as
$m\rightarrow\infty.$ If the end $K$ contains at least one regular
chain, then $K$ will be called {\it regular}. We say that a bounded
domain $D$ in ${\Bbb R}^n$ is {\it regular}, if $D$ can be
quasiconformally mapped to a domain with a locally quasiconformal
boundary whose closure is a compact in ${\Bbb R}^n,$ and, besides
that, every prime end in $D$ is regular. Note that space
$\overline{D}_P=D\cup E_D$ is metric, which can be demonstrated as
follows. If $g:D_0\rightarrow D$ is a quasiconformal mapping of a
domain $D_0$ with a locally quasiconformal boundary onto some domain
$D,$ then for $x, y\in \overline{D}_P$ we put:
\begin{equation}\label{eq5}
\rho(x, y):=|g^{\,-1}(x)-g^{\,-1}(y)|\,,
\end{equation}
where the element $g^{\,-1}(x),$ $x\in E_D,$ is to be understood as
some (single) boundary point of the domain $D_0.$ The specified
boundary point is unique and well-defined by~\cite[Theorem~2.1,
Remark~2.1]{IS}, cf.~\cite[Theorem~4.1]{Na$_2$}. It is easy to
verify that~$\rho$ in~(\ref{eq5}) is a metric on $\overline{D}_P,$
and that the topology on $\overline{D}_P,$ defined by such a method,
does not depend on the choice of the map $g$ with the indicated
property.

We say that a sequence $x_m\in D,$ $m=1,2,\ldots,$ converges to a
prime end of $P\in E_D$ as $m\rightarrow\infty, $ write
$x_m\rightarrow P$ as $m\rightarrow\infty,$ if for any $k\in {\Bbb
N}$ all elements $x_m$ belong to $d_k$ except for a finite number.
Here $d_k$ denotes a sequence of nested domains corresponding to the
definition of the prime end $P.$ Note that for a homeomorphism of a
domain $D$ onto $D^{\,\prime},$ the end of the domain $D$ uniquely
corresponds to some sequence of nested domains in the image under
the mapping.

\medskip
Consider the following Cauchy problem:
\begin{equation}\label{eq2C}
f_{\overline{z}}=\mu(z)\cdot f_z\,,
\end{equation}
\begin{equation}\label{eq1A}
\lim\limits_{\zeta\rightarrow P}{\rm
Re\,}f(\zeta)=\varphi(P)\qquad\forall\,\, P\in E_D\,,
\end{equation}
where $\varphi: E_D\rightarrow {\Bbb R}$  is a predefined continuous
function. In what follows, we assume that $D$ is some simply
connected domain in ${\Bbb C}.$ The solution of the
problem~(\ref{eq2C})--(\ref{eq1A}) is called {\it regular,} if one
of two is fulfilled: or $f(z)=const$ in $D,$ or $f$ is an open
discrete $W_{\rm loc}^{1, 1}(D)$-mapping such that $J(z, f)\ne 0$
for almost any $z\in D.$

\medskip
Given $z_0\in D,$ a function $\varphi: E_D\rightarrow {\Bbb R},$ a
function $\Phi:\overline{{\Bbb R^{+}}}\rightarrow \overline{{\Bbb
R^{+}}}$ and a function $\mathcal{M}(\Omega)$ of open sets
$\Omega\subset D,$ we denote by $\frak{F}^{\mathcal{M}}_{\varphi,
\Phi, z_0}(D)$ the class of all regular solutions
$f:D\rightarrow{\Bbb C}$ of the Cauchy
problem~(\ref{eq2C})--(\ref{eq1A}) that satisfy the condition ${\rm
Im}\,f(z_0)=0$ and, in addition,
\begin{equation}\label{eq1D}
\int\limits_{\Omega}\Phi(K_{\mu}(z))\cdot\frac{dm(z)}{(1+|z|^2)^2}\leqslant
\mathcal{M}(\Omega)
\end{equation}
for any open set $\Omega\subset D.$ The following statement
generalizes~\cite[Theorem~2]{Dyb} to the case of arbitrary simply
connected domains.

\medskip
\begin{theorem}\label{th2A}
{\sl Let $D$ be some simply connected domain in ${\Bbb C},$ and let
$\Phi:\overline{{\Bbb R^{+}}}\rightarrow \overline{{\Bbb R^{+}}}$ be
a continuous increasing convex function, which satisfies the
condition
$$
\int\limits_{\delta}^{\infty}\frac{d\tau}{\tau\Phi^{\,-1}(\tau)}=\infty
$$
for some $\delta>\Phi(0).$ Assume that the function $\mathcal{M}$ is
bounded, and the function $\varphi$ in~(\ref{eq1A}) is continuous.
Then the family $\frak{F}^{\mathcal{M}}_{\varphi, \Phi, z_0}(D)$ is
compact in~$\overline{D}_P.$}
\end{theorem}

\section{Convergence theorems for mappings with upper estimates for modulus distortion }

The proof of the main result is based on the theorems on the global
behavior of mappings satisfying the weight Poletsky inequality.
Results of a similar type in some other situations have been
obtained by us before, see, for example,  \cite{Sev$_1$},
\cite{SevSkv$_1$} and \cite{SevSkv$_3$}. The case we will consider
below concerns regular domains and mappings with one normalization
condition. This case is considered for the first time in this degree
of generality.

\medskip
Given $p\geqslant 1,$ $M_p$ denotes the $p$-modulus of a family of
paths, and the element $dm(x)$ corresponds to a Lebesgue measure in
${\Bbb R}^n,$ $n\geqslant 2,$ see~\cite{Va}. In what follows, we
usually write $M(\Gamma)$ instead of $M_n(\Gamma).$ Everywhere
below, unless otherwise stated, the boundary and the closure of a
set are understood in the sense of an extended Euclidean space
$\overline{{\Bbb R}^n}.$ Let $x_0\in\overline{D},$ $x_0\ne\infty,$
$$S(x_0,r) = \{
x\,\in\,{\Bbb R}^n : |x-x_0|=r\}\,, S_i=S(x_0, r_i)\,,\quad
i=1,2\,,$$
\begin{equation}\label{eq6} A=A(x_0, r_1, r_2)=\{ x\,\in\,{\Bbb R}^n :
r_1<|x-x_0|<r_2\}\,.
\end{equation}
Everywhere below, unless otherwise stated, the closure
$\overline{A}$ and the boundary $\partial A $ of the set $A$ are
understood in the topology of the space $\overline{{\Bbb R}^n}={\Bbb
R}^n\cup\{\infty\}.$ Let $Q:{\Bbb R}^n\rightarrow {\Bbb R}^n$ be a
Lebesgue measurable function satisfying the condition $Q(x)\equiv 0$
for $x\in{\Bbb R}^n\setminus D,$ and let $p\geqslant 1.$ Given sets
$E$ and $F$ and a given domain $D$ in $\overline{{\Bbb R}^n}={\Bbb
R}^n\cup \{\infty\},$ we denote by $\Gamma(E, F, D)$ the family of
all paths $\gamma:[0, 1]\rightarrow \overline{{\Bbb R}^n}$ joining
$E$ and $F$ in $D,$ that is, $\gamma(0)\in E,$ $\gamma(1)\in F$ and
$\gamma(t)\in D$ for all $t\in (0, 1).$ According
to~\cite[Сhap.~7.6]{MRSY}, a mapping $f:D\rightarrow \overline{{\Bbb
R}^n}$ is called a {\it ring $Q$-mapping at the point $x_0\in
\overline{D}\setminus \{\infty\}$ with respect to $p$-modulus}, if
the condition
\begin{equation} \label{eq3*!gl0}
M_p(f(\Gamma(S_1, S_2, D)))\leqslant \int\limits_{A\cap D} Q(x)\cdot
\eta^p (|x-x_0|)\, dm(x)
\end{equation}
holds for all $0<r_1<r_2<d_0:=\sup\limits_{x\in D}|x-x_0|$ and all
Lebesgue measurable functions $\eta:(r_1, r_2)\rightarrow [0,
\infty]$ such that
\begin{equation}\label{eq*3gl0}
\int\limits_{r_1}^{r_2}\eta(r)\,dr\geqslant 1\,.
\end{equation}
The mapping $f:D\rightarrow \overline{{\Bbb R}^n}$ is called a {\it
ring $Q$-mapping in $\overline{D}\setminus\{\infty\}$ with respect
to $p$-modulus} if~(\ref{eq3*!gl0}) holds for any~$x_0\in
\overline{D}\setminus\{\infty\}.$ This definition can also be
applied to the point~$x_0=\infty$ by inversion:
$\varphi(x)=\frac{x}{|x|^2},$ $\infty\mapsto 0.$ In what follows,
$h$ denotes the so-called chordal metric defined by the equalities
\begin{equation}\label{eq1E}
h(x,y)=\frac{|x-y|}{\sqrt{1+{|x|}^2} \sqrt{1+{|y|}^2}}\,,\quad x\ne
\infty\ne y\,, \quad\,h(x,\infty)=\frac{1}{\sqrt{1+{|x|}^2}}\,.
\end{equation}
For a given set $E\subset\overline{{\Bbb R}^n},$ we set
\begin{equation}\label{eq9C}
h(E):=\sup\limits_{x,y\in E}h(x, y)\,,
\end{equation}
The quantity $h(E)$ in~(\ref{eq9C}) is called the {\it chordal
diameter} of the set $E.$ For given sets $A, B\subset
\overline{{\Bbb R}^n},$ we put
$h(A, B)=\inf\limits_{x\in A, y\in B}h(x, y),$
where $h$ is a chordal metric defined in~(\ref{eq1E}).

\medskip
Let $I$ be a fixed set of indices and let $D_i,$ $i\in I,$ be some
sequence of domains. Following~\cite[Sect.~2.4]{NP}, we say that a
family of domains $\{D_i\}_{i\in I}$ is {\it equi-uniform with
respect to $p$-modulus} if for any $r> 0$ there exists a number
$\delta> 0$ such that the inequality
\begin{equation}\label{eq17***}
M_p(\Gamma(F^{\,*},F, D_i))\geqslant \delta
\end{equation}
holds for any $i\in I$ and any continua $F, F^*\subset D$ such that
$h(F)\geqslant r$ and $h(F^{\,*})\geqslant r.$

\medskip
Given a Lebesgue measurable function $Q:{\Bbb R}^n\rightarrow [0,
\infty]$ and a point $x_0\in {\Bbb R}^n$ we set
\begin{equation}\label{eq10}
q_{x_0}(t)=\frac{1}{\omega_{n-1}r^{n-1}} \int\limits_{S(x_0,
t)}Q(x)\,d\mathcal{H}^{n-1}\,,
\end{equation}
where $\mathcal{H}^{n-1}$ denotes $(n-1)$-dimensional Hausdorff
measure. The following lemma was formulated and proved
in~\cite{Sev$_5$}; however, for completeness of presentation, we
present it in full.

\medskip
\begin{lemma}\label{lem1}
{\sl\, Let $1\leqslant p\leqslant n,$ and let $\Phi:[0,
\infty]\rightarrow [0, \infty] $ be a strictly increasing convex
function such that the relation
\begin{equation}\label{eq2} \int\limits_{\delta_0}^{\infty}
\frac{d\tau}{\tau\left[\Phi^{-1}(\tau)\right]^{\frac{1}{p-1}}}=
\infty
\end{equation}
holds for some $\delta_0>\tau_0:=\Phi(0).$ Let $\frak{Q}$ be a
family of functions $Q:{\Bbb R}^n\rightarrow [0, \infty]$ such that
\begin{equation}\label{eq5A}
\int\limits_D\Phi(Q(x))\frac{dm(x)}{\left(1+|x|^2\right)^n}\
\leqslant M_0<\infty
\end{equation}
for some $0<M_0<\infty.$ Now, for any $0<r_0<1$ and for every
$\sigma>0$ there exists $0<r_*=r_*(\sigma, r_0, \Phi)<r_0$ such that
$$
\int\limits_{\varepsilon}^{r_0}\frac{dt}{t^{\frac{n-1}{p-1}}q^{\frac{1}{p-1}}_{x_0}(t)}\geqslant
\sigma\,,\qquad \varepsilon\in (0, r_*)\,,
$$
for any $Q\in \frak{Q}.$ }
\end{lemma}

\medskip
\begin{proof}
Using the substitution of variables $t=r/r_0, $ for any
$\varepsilon\in (0, r_0) $ we obtain that
\begin{equation}\label{eq34}
\int\limits_{\varepsilon}^{r_0}\frac{dr}{r^{\frac{n-1}{p-1}}q^{\frac{1}{p-1}}_{x_0}(r)}
\geqslant
\int\limits_{\varepsilon}^{r_0}\frac{dr}{rq^{\frac{1}{p-1}}_{x_0}(r)}
=\int\limits_{\varepsilon/r_0}^1\frac{dt}{tq^{\frac{1}{p-1}}_{x_0}(tr_0)}
=\int\limits_{\varepsilon/r_0}^1\frac{dt}{t\widetilde{q}^{\frac{1}{p-1}}_{0}(t)}\,,
\end{equation}
where $\widetilde{q}_0(t)$ is the average integral value of the
function $\widetilde{Q}(x):=Q(r_0x+x_0)$ over the sphere $|x|=t,$
see the ratio~(\ref{eq10}). Then, according to~\cite[Lemma~3.1]{RS},
\begin{equation}\label{eq35}
\int\limits_{\varepsilon/r_0}^1\frac{dt}{t\widetilde{q}^{\frac{1}{p-1}}_{0}(t)}\geqslant
\frac{1}{n}\int\limits_{eM_*\left(\varepsilon/r_0\right)}^{\frac{M_*\left(\varepsilon/r_0\right)
r_0^n}{\varepsilon^n}}\frac{d\tau}
{\tau\left[\Phi^{-1}(\tau)\right]^{\frac{1}{p-1}}}\,,
\end{equation}
where
$$M_*\left(\varepsilon/r_0\right)=
\frac{1}{\Omega_n\left(1-\left(\varepsilon/r_0\right)^n\right)}
\int\limits_{A\left(0, \varepsilon/r_0, 1\right)} \Phi\left(Q(r_0x
+x_0)\right)\,dm(x)=$$
$$=
\frac{1}{\Omega_n\left(r_0^n-\varepsilon^n\right)}
\int\limits_{A\left(x_0, \varepsilon, r_0\right)}
\Phi\left(Q(x)\right)\,dm(x)$$
and~$A(x_0, \varepsilon, r_0)$ is defined in~(\ref{eq6}) for
$r_1:=\varepsilon$ and $r_2:=r_0.$ Observe that $|x|\leqslant
|x-x_0|+ |x_0|\leqslant r_0+|x_0|$ for any~$x\in A(x_0, \varepsilon,
r_0).$ Thus
$$M_*\left(\varepsilon/r_0\right)\leqslant \frac{\beta(x_0)}
{\Omega_n\left(r_0^n-\varepsilon^n\right)}\int\limits_{A(x_0,
\varepsilon, r_0)}
\Phi(Q(x))\frac{dm(x)}{\left(1+|x|^2\right)^n}\,,$$
where $\beta(x_0)=\left(1+(r_0+|x_0|)^2\right)^n.$ Therefore,
$$M_*\left(\varepsilon/r_0\right)\leqslant \frac{2\beta(x_0)}{\Omega_n r^n_0}M_0$$
for $\varepsilon\leqslant r_0 /\sqrt[n]{2},$ where $M_0$ is a
constant in~(\ref{eq5A}).
Observe that
$$M_*\left(\varepsilon/r_0\right)>\Phi(0)>0\,,$$
because $\Phi$ is increasing. Now, by~(\ref{eq34}) and~(\ref{eq35})
we obtain that
\begin{equation}\label{eq12}\int\limits_{\varepsilon}^{r_0}\frac{dr}{r^{\frac{n-1}{p-1}}q^{\frac{1}{p-1}}_{x_0}(r)}
\geqslant
\frac{1}{n}\int\limits_{\frac{2\beta(x_0)M_0e}{\Omega_nr^n_0}}
^{\frac{\Phi(0)r^n_0}{\varepsilon^n}}\frac{d\tau}
{\tau\left[\Phi^{\,-1}(\tau)\right]^{\frac{1}{p-1}}}\,.
\end{equation}
The desired conclusion follows from~(\ref{eq12})
and~(\ref{eq2}).~$\Box$
\end{proof}

\medskip
Definitions of a condenser and its capacity may be found, for
example, in~\cite[Sect.~10, Ch.~II]{Ri}. The following statement
holds (see, e.g., \cite[item~(8.9)]{Ma}).

\medskip
\begin{proposition}\label{pr1A}
{\sl Let $E=(A, C)$ be a condenser and let $1<p<n.$ Then
$$
{\rm cap}_p\,E\geqslant n{\Omega}^{\frac{p}{n}}_n
\left(\frac{n-p}{p-1}\right)^{p-1}\left[m(C)\right]^{\frac{n-p}{n}}\,,
$$
where $m(C)$ is the Lebesgue measure of $C.$}
\end{proposition}

\medskip
Consider another auxiliary family of mappings. For $p\geqslant 1,$
given numbers $\delta>0,$ $0<M_0<\infty,$ a domain $D\subset {\Bbb
R}^n,$ $n\geqslant 2,$ and a given strictly increasing convex
function $\Phi\colon\overline{{\Bbb
R}^{+}}\rightarrow\overline{{\Bbb R}^{+}}$ denote by
$\frak{A}_{\Phi, p, \delta, M_0}(D)$ the family all open discrete
mappings $f:D\rightarrow {\Bbb R}^n$ satisfying
relations~(\ref{eq3*!gl0})--(\ref{eq*3gl0}) with some $Q=Q_f$ in
$\overline{D}$ with respect to $p$-modulus. The following statement
is true.

\medskip
\begin{lemma}\label{lem4}
{\sl\, Let $p\in (n-1, n),$ and let $\delta_0>\tau_0:=\Phi(0)$ be
such that the condition~(\ref{eq5A}) holds. Now the family
$\frak{A}_{\Phi, p, \delta, M_0}(D)$ is equicontinuous in $D.$ }
\end{lemma}

\medskip
In the lemma~\ref{lem4} the equicontinuity of the family of mappings
$\frak{A}_{\Phi, p, \delta, M_0}(D)$ should be understood as between
the spaces $(D, d)$ and $({\Bbb R}^n, d),$ where $d$ is the
Euclidean metric.

\medskip
\begin{proof} Let us fix $x_0\in D.$
Let us use the approach used to prove Lemma~2.4 in~~\cite{GSS$_2$}.
Let $0<r_0<{\rm dist\,}(x_0,\,\partial D).$ Consider the condenser
$E=(A, C),$ where $A=B(x_0, r_0),$ $C=\overline{B(x_0,
\varepsilon)}.$ In this case, according to~\cite[Lemma~1]{SalSev}
$$
{\rm cap}_p\, f(E)\leqslant\ \frac{\omega_{n-1}}{I^{\,p-1}}\,,
$$
where
$$
I=I(x_0, \varepsilon, r_0)=\int\limits_{\varepsilon}^{r_0}\
\frac{dr}{r^{\frac{n-1}{p-1}}q_{x_0}^{\frac{1}{p-1}}(r)}
$$
and $f(E)=\left(f\left(B(x_0, r_0)\right), f\left(\overline{B(x_0,
\varepsilon)}\right)\right).$ By Lemma~\ref{lem1} and due to the
condition~(\ref{eq2}) there is a function
$\alpha=\alpha(\varepsilon)$ and a number
$0<\varepsilon^{\,\prime}_0<r_0$ such that
$\alpha(\varepsilon)\rightarrow 0$ as $\varepsilon\rightarrow 0$
and, in addition,
$${\rm cap}_p\,f(E)\leqslant \alpha(\varepsilon)$$
for any $\varepsilon\in (0, \varepsilon^{\,\prime}_0)$ and
$f\in\frak{A}_{\Phi, p, \delta, M_0}(D).$
By Proposition~\ref{pr1A} we obtain that
$$
\alpha(\varepsilon)\geqslant{\rm cap}_p\,f(E)\geqslant
n{\Omega}^{\frac{p}{n}}_n
\left(\frac{n-p}{p-1}\right)^{p-1}\left[m(f(C))\right]^{\frac{n-p}{n}}\,,
$$
where $m(C)$ denotes the Lebesgue measure of $C.$ In other words,
\begin{equation*}
m(f(C))\leqslant \alpha_1(\varepsilon)\,,
\end{equation*}
where $\alpha_1(\varepsilon)\rightarrow 0$ as
$\varepsilon\rightarrow 0.$ From the last relation it follows that
there exists a number $\varepsilon_1\in (0, 1),$ such that
\begin{equation}\label{eqroughb}
m(f(C))\leqslant 1\,,
\end{equation}
where $C=\overline{B(x_0, \varepsilon_1)}.$

\medskip
Let $E_1=(A_1, C_{\varepsilon}),$ $A_1=B(x_0, \varepsilon_1),$ and
$C_{\varepsilon}=\overline{B(x_0, \varepsilon)},$ $\varepsilon\in
(0, \varepsilon_1).$ By Lemma~\ref{lem1} there exists a function
$\alpha_2(\varepsilon)$ and a number
$0<\varepsilon^{\,\prime}_0<\varepsilon_1$ such that
$$
{\rm cap}_p\,f(E_1)\leqslant \alpha_2(\varepsilon)
$$
for any $\varepsilon\in (0, \varepsilon^{\,\prime}_0),$ where
$\alpha_2(\varepsilon)\rightarrow 0$ as $\varepsilon\rightarrow 0.$
On the other hand, by~\cite[Proposition~6]{Kr},
cf.~\cite[Lemma~5.9]{MRV}
\begin{equation}\label{eq1G} \left(c_1\frac{\left(d(f(\overline{B(x_0,
\varepsilon)}))\right)^p} {\left(m(f(B(x_0,
\varepsilon_1)))\right)^{1-n+p}}\right)^{\frac{1}{n-1}} \leqslant
{\rm cap\,}_p\,f\left(E_1\right)\leqslant \alpha_2(\varepsilon)\,.
\end{equation}
By~(\ref{eqroughb}) and~(\ref{eq1G}), we obtain that
\begin{equation}\label{eq1C}
d(f(\overline{B(x_0, \varepsilon)})) \leqslant
\alpha_3(\varepsilon)\,,
\end{equation}
where $\alpha_3(\varepsilon)\rightarrow 0$ as
$\varepsilon\rightarrow 0.$ The proof of Lemma~\ref{lem4} is
complete because the mapping $f\in \frak{A}_{\Phi, p, \delta,
M_0}(D)$ is arbitrary.~$\Box$
\end{proof}

\medskip
Given $p\geqslant 1,$ numbers $\delta>0,$ $0<M_0<\infty,$ a domain
$D\subset {\Bbb R}^n,$ $n\geqslant 2,$ a point $a\in D$ and a
strictly increasing convex function $\Phi\colon\overline{{\Bbb
R}^{+}}\rightarrow\overline{{\Bbb R}^{+}}$ denote by
$\frak{F}_{\Phi, a, p, \delta, M_0}(D)$ the family of all
homeomorphisms $f:D\rightarrow \overline{{\Bbb R}^n}$ satisfying
(\ref{eq3*!gl0})--(\ref{eq*3gl0}) in $\overline{D}$ for some $Q=Q_f$
such that $h(f(a),
\partial f(D))\geqslant\delta,$ $h(\overline{{\Bbb R}^n}\setminus
f(D))\geqslant \delta$ and, in addition, (\ref{eq5A}) holds.

\medskip
\begin{theorem}\label{th1} {\sl\, Let $p\in (n-1, n],$ let $D$ be regular,
and let $D_f^{\,\prime}=f(D)$ be bounded domains with a locally
quasiconformal boundary which are equi-uniform with respect to
$p$-modulus over all $f\in\frak{F}_{\Phi, a, p, \delta, M_0}(D).$ If
there is $\delta_0>\tau_0:=\Phi(0)$ such that~(\ref{eq2}) holds,
then any $f\in\frak{F}_{\Phi, a, p, \delta, M_0}(D)$ has a
continuous extension~$\overline{f}: \overline{D}_P\rightarrow
\overline{{\Bbb R}^n}$ and, in addition, the family $\frak{F}_{\Phi,
a, p, \delta, M_0}(\overline{D})$ of all extended mappings
$\overline{f}: \overline{D}_P\rightarrow \overline{{\Bbb R}^n}$ is
equicontinuous in $\overline{D}_P.$
  }
\end{theorem}

\medskip
\begin{remark}\label{rem2}
In Theorem~\ref{th1}, the equicontinuity should be understood in the
sense of mappings acting between the spaces~$(X, d)$ and
$\left(X^{\,\prime}, d^{\,\prime}\right),$ where $X=\overline{D}_P$
is the complement of the domain $D$ by its prime ends, and $d$ is
one of the possible metrics that correspond to the topological space
$\overline{D}_P$ in~(\ref{eq5}). In addition,
$X^{\,\prime}=\overline{{\Bbb R}^n}$ and $d^{\,\prime}=h$ is a
chordal (spherical) metric.

An example of a family of plane mappings $f_n(z)=z^n,$ $n=1,2,\ldots
,$ $z\in {\Bbb D},$ indicates the inaccuracy of the analogue of
Theorem~\ref{th1} for mappings with branching, in particular, this
theorem is not true under the normalization condition $f_n(0)=0,$
$n=1,2,\ldots .$
\end{remark}

\medskip
\begin{proof}
Put $f\in \frak{F}_{\Phi, A, p, \delta, M_0}(D)$ and $Q=Q_f(x).$
Given $x\in {\Bbb R}^n$ we set
$$Q^{\,\prime}(x)=\begin{cases}Q(x), &x\in D, Q(x)\geqslant 1\\ 1, & x\in D, Q(x)<1\\
1,& x\not\in D\end{cases}\,.$$
Observe that the function $Q^{\,\prime}(x)$ satisfies the
relation~(\ref{eq5A}) up to some constant. Indeed,
$$
\int\limits_D\Phi(Q^{\,\prime}(x))\frac{dm(x)}{\left(1+|x|^2\right)^n}=
\int\limits_{\{x\in D: Q(x)< 1
\}}\Phi(Q^{\,\prime}(x))\frac{dm(x)}{\left(1+|x|^2\right)^n}+$$$$+
\int\limits_{\{x\in D: Q(x)\geqslant 1\}
}\Phi(Q^{\,\prime}(x))\frac{dm(x)}{\left(1+|x|^2\right)^n}\leqslant
\delta+\Phi(1)\int\limits_{{\Bbb
R}^n}\frac{dm(x)}{\left(1+|x|^2\right)^n}=M^{\,\prime}_0<\infty\,.$$
In this case, by Lemma~\ref{lem1}
$$\int\limits_{\varepsilon}^{r_0}\frac{dt}{t^{\frac{n-1}{p-1}}q^{\,\prime\frac{1}{p-1}}_{x_0}(t)}
\rightarrow\infty $$
for any $0<r_0<1$ and $\varepsilon\rightarrow 0,$ where
$q^{\,\prime}_{x_0}(t)=\frac{1}{\omega_{n-1}r^{n-1}}
\int\limits_{S(x_0, t)}Q^{\,\prime}(x)\,d\mathcal{H}^{n-1}\,.$
Besides that,
$\int\limits_{\varepsilon}^{r_0}\frac{dt}{t^{\frac{n-1}{p-1}}q^{\,\prime\frac{1}{p-1}}_{x_0}(t)}<\infty$
for any $\varepsilon\in (0, r_0),$ because
$q^{\,\prime}_{x_0}(t)\geqslant 1$ for almost any $t\in (0, r_0).$
Set
$\psi(t)=\frac{1}{t^{\frac{n-1}{p-1}}q_{x_0}^{\frac{1}{p-1}}(t)}$
for $t\in (0, r_0).$ Now, by the Fubini theorem, we obtain that
\begin{equation}\label{eq10B}
\frac{\omega_{n-1}}{I^{p-1}}=\frac{1}{I^p}\int\limits_{\varepsilon<|x-x_0|<r_0}
Q(x)\cdot \psi^p(|x-x_0|)\ dm(x)\rightarrow 0, \quad
\varepsilon\rightarrow 0\,,
\end{equation}
where $I:=I(\varepsilon, r_0)=\int\limits_{\varepsilon}^{r_0}\
\frac{dt}{t^{\frac{n-1}{p-1}}q_{x_0}^{\,\prime\frac{1}{p-1}}(r)}.$
In this case, by~\cite[Lemma~3]{Sev$_5$} a mapping
$f\in\frak{F}_{\Phi, A, p, \delta, M_0}(D)$ has a continuous
extension to $\overline{D}_P.$ In particular, the strong
accessibility of the boundary of $D_f^{\,\prime}=f(D)$ with respect
to $p$-modulus follows from~\cite[Remark]{SevSkv$_1$}. Observe that
$f\in \frak{F}_{\Phi, a, p, \delta, M_0}(\overline{D})$ does not
take the point infinity for $p\ne n$ (see, e.g., \cite[Lemmas 2.6
and 3.1]{GSS$_1$}). Now, the equicontinuity of $\frak{F}_{\Phi, a,
p, \delta, M_0}(\overline{D})$ inside $D$ follows by Theorem~4.1
in~\cite{RS} fpr $p=n$ and Lemma~\ref{lem4} for $p<n.$

\medskip
We prove the equicontinuity of the family $\frak{F}_{\Phi, a, p,
\delta, M_0}(\overline{D})$ in $E_D:=\overline{D}_P\setminus D.$ Let
us assume the opposite, namely that there are $\varepsilon_*>0,$
$P_0\in E_D,$ a sequence $x_m\in \overline{D}_P,$ $x_m\rightarrow
P_0$ as $m\rightarrow \infty,$ and a mapping  $f_m\in
\frak{F}_{\Phi, a, p, \delta, M_0}(\overline{D})$ such that

\begin{equation}\label{eq1F}
h(f_m(x_m), f_m(P_0))\geqslant\varepsilon_*\,,\quad m=1,2,\ldots\,.
\end{equation}
Since $f_m$ as a continuous extension at $P_0,$ we may assume that
$x_m\in D$ and, in addition, there is a sequence $x^{\,\prime}_m\in
\overline{D}_P,$ $x^{\,\prime}_m\rightarrow P_0$ as
$m\rightarrow\infty,$ such that
\begin{equation}\label{eq1*}
h(f_m(x_m), f_m(x^{\,\prime}_m))\geqslant\varepsilon_*/2\,,\quad
m=1,2,\ldots\,.
\end{equation}
Let $d_m$ be a sequence of cuts~$\sigma_m$ corresponding to $P_0$
such that $\sigma_m\subset S(x_0, r_m),$ $x_0\in\partial D$ and
$r_m\rightarrow 0$ as $m\rightarrow\infty$
(see~\cite[Lemma~2]{KR$_1$}). Without loss of a generality, we may
assume that $x_m, x^{\,\prime}_m\in d_m.$ Let $\gamma_m$ be a path
joining $x_m$ and $x^{\,\prime}_m$ in $d_m.$

\medskip
Since the domain $D$ is regular, the space $\overline{D}_P$ contains
at least two prime ends $P_1$ and $P_2 \in E_D.$ Let $P_1\subset
E_D$ be a prime end that does not coincide with $P_0.$ Suppose that
$G_m,$ $m=1,2,\ldots,$ is a sequence of domains that corresponds to
a prime end $P_1.$ Since the mapping $f_m$ has a continuous
extension on $\overline{D}_P$ for any $m=1,2,\ldots ,$ we may choose
a sequence $\zeta_m\in G_m,$ $\zeta_m\rightarrow P_1$ as
$m\rightarrow\infty,$ such that $h(f_m(\zeta_m),
f_m(P_1))\rightarrow 0$ as $m\rightarrow\infty.$ Note that
\begin{equation}\label{eq3D}
h(f_m(a), f_m(\zeta_m))\geqslant h(f_m(a), f_m(P_1))-h(f_m(\zeta_m),
f_m(P_1))\geqslant \delta/2\,,
\end{equation}
for any $m\geqslant m_0$ and some $m_0\in {\Bbb N}.$ We construct a
sequence of continua $K_m,$ $m=1,2, \ldots$ as follows. We join the
points $\zeta_1$ and $a$ by an arbitrary path in $D,$ which we
denote by $K_1.$ Next, we join the points $\zeta_2$ and $\zeta_1$ by
a path $K_1^{\prime},$ in $G_1.$ Combining the paths $K_1$ and
$K_1^{\prime},$ we obtain a path $K_2,$ joining the points $a$ and
$\zeta_2.$ And so on. Suppose that at some step we have a path
$K_m,$ that join the points $\zeta_m$ and $a.$ Join the points
$\zeta_{m+1}$ and $\zeta_m$ with a path $K_m^{\,\prime},$ which lies
in $G_m.$ Combining the paths $K_m$ and $K_m^{\,\prime},$ we obtain
a path $K_{m+1}.$ And so on (see Figure~\ref{fig2}).
\begin{figure}[h]
\centerline{\includegraphics[scale=0.7]{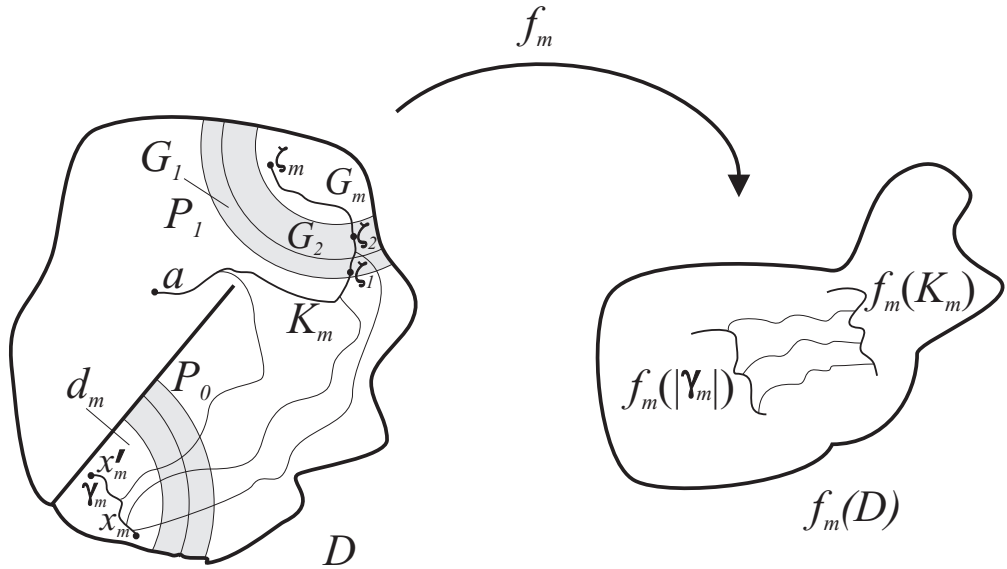}} \caption{To
proof of Theorem~\ref{th1}}\label{fig2}
\end{figure}
We show that there is a number $m_1\in {\Bbb N}$ such that
\begin{equation}\label{eq4D}
d_m\cap K_m=\varnothing\quad\forall\quad m\geqslant m_1\,.
\end{equation}
We prove this from the opposite, namely, suppose that~(\ref{eq4D})
does not hold. Then there is an increasing sequence of
numbers~$m_k\rightarrow\infty,$ $k\rightarrow\infty,$ and
points~$\xi_k\in K_{m_k}\cap d_{m_k},$ $m=1,2,\ldots,\,.$ Then
$\xi_k \rightarrow P_0$ as $k\rightarrow\infty.$

\medskip
Note that two cases are possible: either all elements $\xi_k$ belong
to $D\setminus G_1$ for $k = 1,2, \ldots ,$  or there is a number
$k_1$ such that $\xi_{k_1}\in G_1.$ In the second case, consider the
sequence~$\xi_k,$ $k> k_1.$ Note that two cases are possible: or
$\xi_k$ for $k>k_1$ belong to $D \setminus G_2,$ or there is
$k_2>k_1$ such that $\xi_{k_2}\in G_2.$ In the second case, consider
the sequence~$\xi_k,$ $k>k_2,$ and so on. Assume that the element
$\xi_{k_{l-1}}\in G_{l-1}$ is already constructed. Note that two
cases are possible: either $\xi_k$ belong to $D\setminus G_l $ for
$k>k_{l-1},$ or there is a number $k_l>k_{l-1}$ such that
$\xi_{k_l}\in G_l, $ and etc. This procedure can be both finite or
infinite, depending on which we have two possible situations:

\medskip
1) or there are numbers $n_0\in {\Bbb N}$ and $l_0\in {\Bbb N}$ such
that that $\xi_k\in D\setminus G_{n_0}$ for all $k>l_0;$

2) or for each  there is an element $\xi_{k_l}$ such that
$\xi_{k_l}\in G_l,$ and the sequence $k_l$ is increasing by $l\in
{\Bbb N}.$

\medskip
Consider each of these cases separately and show that in both of
them we come to a contradiction. Let situation~1) holds. Observe
that all elements of the sequence~$\xi_k$ belong to~$K_ {n_0},$
hence the existence of the subsequence~$\xi_{k_r},$ $r=1,2,\ldots,$
convergent as $r\rightarrow \infty $ to some point $\xi_0\in D.$
However, $\xi_k\in d_{m_k}$ and therefore $\xi_0\in
\bigcap\limits_{m=1}^{\infty} \overline{d_m}\subset
\partial D$ (see~\cite[Proposition~1]{KR$_2$}). The obtained
contradiction indicates the impossibility of the case~1). Suppose
that case~2) holds, then simultaneously $\xi_k\rightarrow P_0$ and
$\xi_k\rightarrow P_1$ as $k\rightarrow\infty.$ Since
space~$\overline{D}_P$ is metric with a metric $\rho$
in~(\ref{eq5}), by the triangle inequality it follows that
$P_1=P_0,$ which contradicts the choice of~$P_1.$ The obtained
contradiction indicates the validity of the relation~(\ref{eq4D}).

\medskip
By the relation~(\ref{eq4D}) and by the definition of
cuts~$\sigma_m\subset S(x_0, r_m),$ we obtain that
$$\Gamma\left(|\gamma_m|, K_m, D\right)>\Gamma(S(x_0, r_m), S(x_0,
\widetilde{\varepsilon_0}), D)\,.$$
Thus
$$f_m(\Gamma\left(|\gamma_m|, K_m, D\right))>f_m(\Gamma(S(x_0, r_m), S(x_0,
\widetilde{\varepsilon_0}), D))\,,$$ whence, by the definition of
the class~$\frak{F}_{\Phi, a, p, \delta, M_0}(\overline{D})$
$$M_p(f_m(\Gamma(|\gamma_m|, K_m, D)))\leqslant$$
\begin{equation}\label{eq39}
\leqslant M_p(f_m(\Gamma(S(x_0, r_m), S(x_0,
\widetilde{\varepsilon_0}), D))\leqslant \int\limits_{A\cap D}
Q_m(x)\cdot \eta^p(|x-x_0|)\, dm(x)\,,
\end{equation}
where $\eta$ is any Lebesgue measurable function
satisfying~(\ref{eq*3gl0}) for $r_1\mapsto r_m$ and $r_2\mapsto
\widetilde{\varepsilon_0},$ in addition, $Q_m:=Q_{f_m}$ corresponds
to the function $Q$ in~(\ref{eq3*!gl0}). Let us to prove the
inequality
\begin{equation}\label{eq3BB}
M_p(f_m(\Gamma(S(x_0, r_m), S(x_0, \widetilde{\varepsilon_0}),
D))\leqslant \frac{\omega_{n-1}}{I_m^{p-1}}\,,
\end{equation}
where $I_m=\int\limits_{r_m}^{\widetilde{\varepsilon_0}}\
\frac{dr}{r^{\frac{n-1}{p-1}}q_{mx_0}^{\frac{1}{p-1}}(r)},$
$q_{mx_0}(t)=\frac{1}{\omega_{n-1}r^{n-1}} \int\limits_{S(x_0,
t)}Q_m(x)\,d\mathcal{H}^{n-1}$ and $Q_m:=Q_{f_m}$ (we set
$Q_m(x)\equiv 0$ for $x\not\in D$). To do this, we will reason
similarly to the proof of Lemma~\ref{lem1} in~\cite{SalSev}. We may
assume that $I\ne 0,$ since~(\ref{eq3BB}) is obviously in this case.
We may also assume that $I\ne \infty,$ because otherwise we may
consider $Q(x)+\delta$ instead of $Q(x)$ in (\ref{eq3BB}), and then
go to the limit as $\delta\rightarrow 0. $ Let $0\ne I\ne \infty.$
Then $q_{x_0}(r)\ne 0 $ for $r\in(r_m, \widetilde{\varepsilon_0}).$
Put
$$ \psi(t)= \left \{\begin{array}{rr}
1/[t^{\frac{n-1}{p-1}}q_{x_0}^{\frac{1}{p-1}}(t)], & t\in (r_m,
\widetilde{\varepsilon_0})\ ,
\\ 0,  &  t\notin (r_m,
\widetilde{\varepsilon_0})\ .
\end{array} \right. $$
By the Fubini theorem
\begin{equation}\label{eq3A}
\int\limits_{A} Q_m(x)\cdot\psi^p(|x-x_0|)\,dm(x)=\omega_{n-1}
I_m\,,
\end{equation}
where $A=A(r_m, \widetilde{\varepsilon_0}, x_0)$ is defined
in~(\ref{eq6}). Observe that  the function $\eta_1(t)=\psi(t)/I,$
$t\in (r_m, \widetilde{\varepsilon_0}),$ satisfies~(\ref{eq*3gl0}).
Now, by~(\ref{eq3*!gl0}) and~(\ref{eq3A}) we obtain the
relation~(\ref{eq3BB}), as required.

Finally, from~(\ref{eq39}), (\ref{eq3BB}) and from Lemma~\ref{lem1}
it follows that
\begin{equation}\label{eq39A}
M_p(f_m(\Gamma(|\gamma_m|, K_m, D)))\leqslant
\frac{\omega_{n-1}}{\left(\int\limits_{r_m}^{\widetilde{\varepsilon_0}}
\frac{dr}{r^{\frac{n-1}{p-1}}q^{\frac{1}{p-1}}_{mx_0}(r)}\right)^{p-1}}\rightarrow
0\,,\quad m\rightarrow\infty\,, \end{equation}
where $q_{mx_0}(t)=\frac{1}{\omega_{n-1}r^{n-1}} \int\limits_{S(x_0,
t)}Q_m(x)\,d\mathcal{H}^{n-1}$ and $Q_m:=Q_{f_m}$ corresponds to the
function $Q$ in~(\ref{eq10}). The relation~(\ref{eq39}) contradicts
the equi-uniformity of the sequence of
domains~$D^{\,\prime}_m:=f_m(D).$ Indeed, $h(f_m(K_m))\geqslant
\delta/2$ according to~(\ref{eq3D}), and
$h(f_m(|\gamma_m|))>\varepsilon_*/2$ by the relation~(\ref{eq1*}).
Hence, since the sequence of domains $D^{\,\prime}_m:=f_m(D)$ is
equi-uniform, we obtain that
$$M_p(f_m(\Gamma(|\gamma_m|, K_m, D)))=M_p(\Gamma(f_m(|\gamma_m|), f_m(K_m), f_m(D)))\geqslant
\delta_*>0$$
for some $\delta_*>0$ and any $m=1,2,\ldots, $ which contradicts the
relation~(\ref{eq39A}). The obtained contradiction indicates the
incorrectness of the assumption in~(\ref{eq1F}). Theorem is
proved.~$\Box$
\end{proof}

\section{Equicontinuity of families of mappings with inverse Poletsky inequality}

In this section we deal with mappings $f:D\rightarrow{\Bbb R}^n$
області $D\subset{\Bbb R}^n,$ $n\geqslant 2.$ The main purpose is to
summarize the results of our previous article~\cite{SSD}. In
particular, it is necessary for proving the key theorems of
compactness of mapping classes from the next section.

\medskip
Let $y_0\in {\Bbb R}^n,$ $0<r_1<r_2<\infty$ and
$$
A=A(y_0, r_1,r_2)=\left\{ y\,\in\,{\Bbb R}^n:
r_1<|y-y_0|<r_2\right\}\,.$$
Let, as above, $M(\Gamma)$ be a conformal modulus of families of
paths $\Gamma$ in ${\Bbb R}^n$ (see e.g.~\cite[гл.~6]{Va}). Let
$f:D\rightarrow{\Bbb R}^n,$ $n\geqslant 2,$ and let $Q:{\Bbb
R}^n\rightarrow [0, \infty]$ be a Lebesgue measurable function such
that $Q(x)\equiv 0$ for $x\in{\Bbb R}^n\setminus f(D).$ Let
$A=A(y_0, r_1, r_2).$ Let $\Gamma_f(y_0, r_1, r_2)$ denotes the
family of all paths $\gamma:[a, b]\rightarrow D$ such that
$f(\gamma)\in \Gamma(S(y_0, r_1), S(y_0, r_2), A(y_0, r_1, r_2)),$
i.e., $f(\gamma(a))\in S(y_0, r_1),$ $f(\gamma(b))\in S(y_0, r_2),$
and $\gamma(t)\in A(y_0, r_1, r_2)$ for any $a<t<b.$ We say that
{\it $f$ satisfies the inverse Poletsky inequality} at $y_0\in f
(D)$ if the relation
\begin{equation}\label{eq2*A}
M(\Gamma_f(y_0, r_1, r_2))\leqslant \int\limits_{A(y_0,r_1,r_2)\cap
f(D)} Q(y)\cdot \eta^n (|y-y_0|)\, dm(y)
\end{equation}
holds for any Lebesgue measurable function $\eta:(r_1,
r_2)\rightarrow [0, \infty]$ such that
\begin{equation}\label{eqA2}
\int\limits_{r_1}^{r_2}\eta(r)\, dr\geqslant 1\,.
\end{equation}

\medskip
The following statement holds.

\medskip
\begin{theorem}\label{th3}
{\sl Let $D\subset {\Bbb R}^n,$ $n\geqslant 2,$ be a domain that has
a weakly flat boundary, and let $D^{\,\prime}\subset {\Bbb R}^n$ be
a regular domain. Suppose that $f$ is an open discrete and closed
mapping of $D$ onto $D^{\,\prime}$ that satisfies the
relation~(\ref{eq2*A}) at each point $y_0\in D^{\,\prime}.$ Suppose
that for each point $y_0\in D^{\,\prime}$ and
$0<r_1<r_2<r_0:=\sup\limits_{y\in D^{\,\prime}}|y-y_0|$ there is a
set $E\subset[r_1, r_2] $ of a positive linear Lebesgue measure such
that the function $Q$ is integrable on $S(y_0, r)$ for almost all
$r\in E.$ Then the mapping $f$ has a continuous extension
$\overline{f}:\overline{D}\rightarrow\overline{D^{\,\prime}}_P,$
moreover $\overline{f}(\overline{D})=\overline{D^{\,\prime}}_P.$ }
\end{theorem}

\medskip
\begin{proof} Fix $x_0\in\partial D.$ It is necessary to show
the possibility of continuous extension of the mapping of $f$ to the
point $x_0.$ Using, if necessary, the M\"{o}bius transformation
$\varphi:\infty\mapsto 0$ and taking into account the invariance of
the modulus $M$ in left part of the relation~(\ref{eq2*A})
(see~\cite[Theorem~8.1]{Va}), we may assume that $x_0\ne \infty.$

\medskip
Assume that the conclusion about the continuous extension of the
mapping of $f$ to the point $x_0$ is not correct. Then any prime end
$P_0\in E_{D^{\,\prime}}$ is not a limit of $f$ at the point $x_0,$
that is, there is a sequence $x_k\rightarrow x_0$ as $k\rightarrow
\infty$ and a number $\varepsilon_0> 0$ such that $\rho(f(x_k),
P_0)\geqslant \varepsilon_0$ for any $k\in{\Bbb N}, $ where $\rho$
is one of metric in~(\ref{eq5}). Since by condition the domain
$D^{\,\prime}$ is regular, it can be mapped to a bounded domain
$D_*$ with a locally quasi-conformal boundary using some
quasi-conformal mapping $g:D^{\,\prime}\rightarrow D_*.$ Note that
the points of the boundary of the domain $D^{\,\prime}$ and the
prime ends of the domain $D_* $ are in one-to-one correspondence,
see, for example, (see~\cite[Theorem~2.1]{IS},
cf.~\cite[Theorem~4.1]{Na$_2$}). In this case, since $\overline
{D}_*$ is a compactum  in ${\Bbb R}^n,$ the metric space
$(\overline{D^{\,\prime}}_P, \rho)$ is compact, as well. Now, we may
assume that $f(x_k)$ converges to some element $P_1\ne P_0$ as
$k\rightarrow\infty.$ Since $f$ has no a limit at $x_0,$ there is at
least one more sequence $y_k,$ $k=1,2\ldots ,$ $y_k\rightarrow x_0$
as $k\rightarrow \infty,$ such that $\rho(f(y_k), P_1)\geqslant
\varepsilon_1$ for any $k\in {\Bbb N}$ and some $\varepsilon_1>0.$
Again, since the metric space $(\overline{D^{\,\prime}}_P, \rho)$ is
compact, we may assume that $f(y_k)\rightarrow P_2 $ as
$k\rightarrow\infty,$ $P_1 \ne P_2,$ $P_2 \in
\overline{D^{\,\prime}}_P.$ Since the mapping $f$ is closed, it
preserves the boundary of the domain, see~\cite[Theorem~3.3]{Vu}.
So, $P_1, P_2\in E_{D^{\,\prime}}.$

\medskip
Let $\sigma_m$ and $\sigma^{\,\prime}_m,$ $m=0,1,2,\ldots, $ are
sequences of cuts which correspond to prime ends $P_1$ and $P_2,$
respectively. Assume that cuts $\sigma_m,$ $m=0,1,2,\ldots, $ belong
to spheres $S(z_0, r_m)$ centered at some point $z_0\in
\partial D^{\,\prime},$ where $r_m\rightarrow 0$ as $m\rightarrow\infty$
(such a sequence $\sigma_m$ exists by~\cite[Lemma~3.1]{IS},
cf.~\cite[Lemma~1]{KR$_2$}). Let $d_m$ and $g_m,$ $m=0,1,2,\ldots, $
be sequences of domains in $D^{\,\prime},$ corresponding to cuts
$\sigma_m$ and $\sigma^{\,\prime}_m,$ respectively. Since the space
$(\overline{D^{\,\prime}}_P, \rho)$ is metric, we may consider that
$d_m$ and $g_m$ are disjoint for any $m=0,1,2,\ldots ,$ in
particular,
\begin{equation}\label{eq4F}
d_0\cap g_0=\varnothing\,.
\end{equation}
Since $f(x_k)$ converges to $P_1$ as $k\rightarrow\infty,$ for any
$m\in {\Bbb N}$ there is $k=k(m)$ such that $f(x_k)\in d_m$ for
$k\geqslant k=k(m).$ By renumbering the sequence $x_k,$ we may
ensure that $f(x_k)\in d_k$ for any natural $k.$ Similarly, we may
assume that $f(y_k)\in g_k $ for any $k\in{\Bbb N}.$ Fix points
$f(x_1)$ and $f(y_1).$ Since by the definition of the prime end
$\bigcap\limits_{k=1}^{\infty}d_k=\bigcap\limits_{l=1}^{\infty}g_l=\varnothing,$
there are numbers $k_1$ and $k_2\in{\Bbb N}$ such that
$f(x_1)\not\in d_{k_1}$ and $f(y_1)\not \in g_{k_2}.$ By the
definition of $d_k$ and $g_k,$ we have that $d_k\subset d_{k_0}$ for
any $k\geqslant k_1$ and $g_k\subset g_{k_2}$ for any $k\geqslant
k_2.$ Now, we obtain that
$$
f(x_1)\not\in d_k\,,\quad f(y_1)\not\in g_k\,, \quad
k\geqslant\max\{k_1, k_2\}\,.
$$
Let $\gamma_k$ be a path joining $f(x_1)$ and $f(x_k)$ in $d_1,$ and
let $\gamma^{\,\prime}_k$ be a path joining $f(y_1)$ and $f(y_k)$ in
$g_1.$ In addition, let $\alpha_k$ and $\beta_k$ are total
$f$-liftings of $\gamma_k$ and $\gamma^{\,\prime}_k$ in $D$ starting
at $x_k$ and $y_k,$ respectively (they exists
by~\cite[Lemma~3.7]{Vu}), see Figure~\ref{fig1A}).
\begin{figure}[h]
\centerline{\includegraphics[scale=0.45]{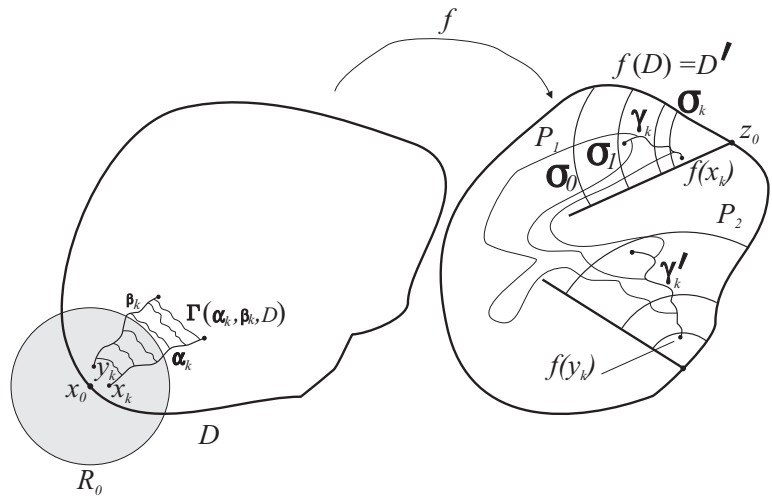}} \caption{To
proof of Theorem~\ref{th3}}\label{fig1A}
\end{figure}
Observe that the points $f(x_1)$ and $ f(y_1)$ cannot have more than
a finite number of pre-images in $D$ under the mapping $f,$
see~\cite[Lemma~3.2]{Vu}. Then there is $R_0>0$ such that
$\alpha_k(1), \beta_k(1)\in D\setminus B(x_0, R_0)$ for any
$k=1,2,\ldots .$ Since the boundary of $D$ is weakly flat, for every
$P>0$ there is $k=k_P\geqslant 1$ such that
\begin{equation}\label{eq7}
M(\Gamma(|\alpha_k|, |\beta_k|, D))>P\qquad\forall\,\,k\geqslant
k_P\,.
\end{equation}
Let us to show that the condition~(\ref{eq7}) contradicts
to~(\ref{eq2*A}). Indeed, let $\gamma\in \Gamma(|\alpha_k|,
|\beta_k|, D).$ Then $\gamma:[0, 1]\rightarrow D,$ $\gamma(0)\in
|\alpha_k|$ and $\gamma(1)\in |\beta_k|.$ In particular,
$f(\gamma(0))\in |\gamma_k|$ and $f(\gamma(1))\in
|\gamma^{\,\prime}_k|.$ In this case, by~(\ref{eq4F})
and~(\ref{eq7}) we obtain that $|f(\gamma)|\cap d_1\ne\varnothing
\ne |f(\gamma)|\cap(D^{\,\prime}\setminus d_1)$ for
$k\geqslant\max\{k_1, k_2\}.$ By~\cite[Theorem~1.I.5.46]{Ku}
$|f(\gamma)|\cap
\partial d_1\ne\varnothing,$ i.e., $|f(\gamma)|\cap S(z_0,
r_1)\ne\varnothing,$ because $\partial d_1\cap D^{\,\prime}\subset
\sigma_1\subset S(z_0, r_1)$ by the definition of the cut
$\sigma_1.$ Let $t_1\in (0,1)$ be such that $f(\gamma(t_1))\in
S(z_0, r_1)$ and $f(\gamma)|_1:=f(\gamma)|_{[t_1, 1]}.$ Without loss
of generality, we may assume that $f(\gamma)|_1\subset {\Bbb
R}^n\setminus B(z_0, r_1).$ Arguing similarly for $f(\gamma)|_1,$ we
may find a point $t_2\in (t_1,1)$ such that $f(\gamma(t_2))\in
S(z_0, r_0).$ Set $f(\gamma)|_2:=f(\gamma)|_{[t_1, t_2]}.$ Now, the
path $f(\gamma)|_2$ is a subpath of $f(\gamma)$ and, moreover,
$f(\gamma)|_2\in \Gamma(S(z_0, r_1), S(z_0, r_0), D^{\,\prime}).$
Without loss of generality, we may assume that $f(\gamma)|_2\subset
B(z_0, r_0).$ Therefore, $\Gamma(|\alpha_k|, |\beta_k|,
D)>\Gamma_f(z_0, r_1, r_0).$
From the last relation, by the minorization of the modulus of
families of paths (see, e.g., \cite[Theorem~1(c)]{Fu}) we obtain
that
$$M(\Gamma(|\alpha_k|, |\beta_k|, D))\leqslant$$
\begin{equation}\label{eq11}
\leqslant M(\Gamma_f(z_0, r_1, r_0))\leqslant M(\Gamma_f(z_0, r_1,
r_0))\leqslant \int\limits_A Q(y)\cdot \eta^n (|y-z_0|)\, dm(y)\,,
\end{equation}
where $A=A(z_0, r_1, r_0)$ and $\eta$ be an arbitrary Lebesgue
measurable function with~(\ref{eqA2}). Bellow we set, as usual,
$a/\infty=0$ for $a\ne\infty,$ $a/0=\infty$ for $a>0$ and
$0\cdot\infty=0$ (see, e.g., \cite[3.I]{Sa}). Set
\begin{equation}\label{eq13}
I=\int\limits_{r_1}^{r_0}\frac{dt}{tq_{z_0}^{1/(n-1)}(t)}\,,
\end{equation}
where
$$
q_{z_0}(r)=\frac{1}{\omega_{n-1}r^{n-1}}\int\limits_{S(z_0,
r)}Q(y)\,d\mathcal{H}^{n-1}(y)$$
and $\omega_{n-1}$ demotes the area of the unit sphere ${\Bbb
S}^{n-1}$ in ${\Bbb R}^n.$ By the assumption, there is a set
$E\subset [r_1, r_0]$ of a positive linear Lebesgue measure such
that $q_{z_0}(t)$ is finite for almost all $t\in E.$ Thus, $I\ne 0$
in~(\ref{eq13}). In this case, the function
$\eta_0(t)=\frac{1}{Itq_{z_0}^{1/(n-1)}(t)}$ satisfies the
relation~(\ref{eqA2}) for $r_2:=r_0.$ Substituting this function
into the right-hand side of the inequality~(\ref{eq11}) and applying
Fubini's theorem, we obtain that
\begin{equation}\label{eq14}
M(\Gamma(|\alpha_k|, |\beta_k|, D))\leqslant
\frac{\omega_{n-1}}{I^{n-1}}<\infty\,.
\end{equation}
The relation~(\ref{eq14}) contradicts the condition~(\ref{eq7}). The
obtained contradiction refutes the assumption that there is no limit
of the mapping $f$ at the point $x_0.$

It remains to check the equality
$\overline{f}(\overline{D})=\overline{D^{\,\prime}}_P.$  Obviously,
$\overline{f}(\overline{D})\subset\overline{D^{\,\prime}}_P.$ We
show that $\overline{D^{\,\prime}}_P\subset
\overline{f}(\overline{D}).$  Indeed, let $y_0\in
\overline{D^{\,\prime}}_P.$ Then either  $y_0\in D^{\,\prime}$ or
$y_0\in E_{D^{\,\prime}}.$ If $y_0\in D^{\,\prime},$ then
$y_0=f(x_0)$ and $y_0\in\overline{f}(\overline{D}),$ because $f$
maps $D$ onto $D^{\,\prime}.$ Finally, let $y_0\in
E_{D^{\,\prime}},$ then by the regularity of the domain
$D^{\,\prime}$ there is a sequence $y_k\in D^{\,\prime}$ such that
$\rho(y_k, y_0)\rightarrow 0$ as $k\rightarrow\infty,$ $y_k=f(x_k)$
and $x_k\in D,$ where $\rho$ is one of possible metrics in
$\overline{D^{\,\prime}}_P.$ Since $\overline{{\Bbb R}^n}$ is a
compact space, we may assume that $x_k\rightarrow x_0,$ where
$x_0\in\overline{D}.$ Observe that $x_0\in
\partial D,$ because $f$ is open. Now
$f(x_0)=y_0\in \overline{f}(\partial D)\subset
\overline{f}(\overline{D}).$ The theorem is completely
proved.~$\Box$
\end{proof}

\medskip
\begin{corollary}\label{cor2}
{\sl\, The statement of Theorem~\ref{th3} remains valid if the
condition for the existence of the set $E\subset[r_1, r_2]$ with the
specified property is replaced by the condition $Q\in
L^1(D^{\,\prime}).$ }
\end{corollary}

\medskip
\begin{proof}
Indeed, by the Fubini theorem (see, e.g.,
\cite[Theorem~8.1.III]{Sa})
$$\int\limits_{B(y_0, r_0)}Q(x)\,dm(x)=\int\limits_0^1\int\limits_{S(y_0, r)}
Q(x)\,d\mathcal{H}^{n-1}dr<\infty\,,$$
because the function $Q$ is integrable in $D^{\,\prime}$ (here we
set $Q(y)\equiv 0$ for $y\not\in D^{\,\prime}$). It follows that $Q$
is integrable on almost all spheres $S(y_0, r),$ $0<r_0<\leqslant
r_0.$ In this case, the desired statement follows from
Theorem~\ref{th3} for $E=[r_1, r_2].$~$\Box$
\end{proof}

\medskip
The following lemma was proved in~\cite[Lemma~2.1]{SSI}.

\medskip
\begin{lemma}\label{lem1A}{\sl\,
Let $D\subset {\Bbb R}^n,$ $n\geqslant 2,$ be a regular domain, and
let $x_m\rightarrow P_1,$ $y_m\rightarrow P_2$ as
$m\rightarrow\infty,$ $P_1, P_2\in \overline{D}_P,$ $P_1\ne P_2.$
Suppose that $d_m, g_m,$ $m=1,2,\ldots,$ are sequences of descending
domains, corresponding to $P_1$ and $P_2,$ $d_1\cap
g_1=\varnothing,$ and $x_0, y_0\in D\setminus (d_1\cup g_1).$ Then
there are arbitrarily large $k_0\in {\Bbb N},$ $M_0=M_0(k_0)\in
{\Bbb N}$ and $0<t_1=t_1(k_0), t_2=t_2(k_0)<1$ for which the
following condition is fulfilled: for each $m\geqslant M_0$ there
are
non-intersecting paths
$$\gamma_{1,m}(t)=\quad\left\{
\begin{array}{rr}
\widetilde{\alpha}(t), & t\in [0, t_1],\\
\widetilde{\alpha_m}(t), & t\in [t_1, 1]\end{array}
\right.\,,\quad\gamma_{2,m}(t)=\quad\left\{
\begin{array}{rr}
\widetilde{\beta}(t), & t\in [0, t_2],\\
\widetilde{\beta_m}(t), & t\in [t_2, 1]\end{array}\,, \right.$$
such that:

1) $\gamma_{1, m}(0)=x_0,$ $\gamma_{1, m}(1)=x_m,$ $\gamma_{2,
m}(0)=y_0$ and $\gamma_{2, m}(1)=y_m;$

2) $|\gamma_{1, m}|\cap \overline{g_{k_0}}=\varnothing=|\gamma_{2,
m}|\cap \overline{d_{k_0}};$

3) $\widetilde{\alpha_m}(t)\in d_{k_0+1}$ for $t\in [t_1, 1]$ and
$\widetilde{\beta_m}(t)\in g_{k_0+1}$ for $t\in [t_2, 1]$ (see
Figure~\ref{fig3}).
\begin{figure}[h]
\centering\includegraphics[width=200pt]{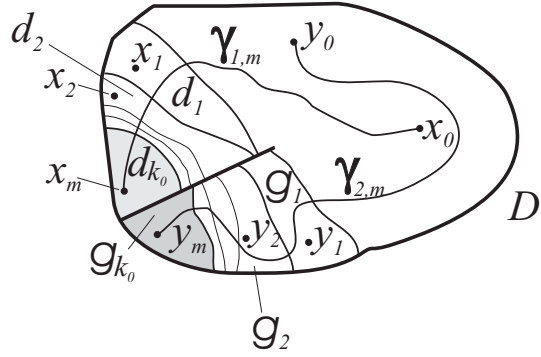} \caption{To
the statement of Lemma~\ref{lem1A}}\label{fig3}
\end{figure}
}
\end{lemma}

\medskip
Let $\gamma_{1, m}$ and $\gamma_{2, m}$ be the paths from
Lemma~\ref{lem1A}, and let $|\gamma_{1, m}|$ and $|\gamma_{2, m}|$
be their images in ${\Bbb R}^n,$ respectively. The following
statement is proved in~\cite[Lemma~2.2]{SSI} for the case of
homeomorphisms. Its proof and content are similar
to~\cite[Lemma~2]{Sev$_4$}, however, for completeness of
presentation, we give it in full in the text.

\medskip
\begin{lemma}\label{lem4A}
{\sl Let $D$ and $D^{\,\prime}$ be domains in ${\Bbb R}^n,$
$n\geqslant 2,$ let $D^{\,\prime}$ be a regular domain, and let $f$
be an open discrete and closed mapping of $D$ onto $D^{\,\prime}$
satisfying the condition~(\ref{eq2*A}) at any point
$y_0\in\overline{D^{\,\prime}}$ with some function $Q\in
L^1(D^{\,\prime}).$ Let $d_m$ be a sequence of decreasing domains
which correspond to cuts $\sigma_m,$ $m=1,2,\ldots, $ lying on the
spheres $S(\overline{x_0}, r_m)$ and such that $\overline{x_0}\in
\partial D^{\,\prime},$ wherein $r_m\rightarrow 0$ as $m\rightarrow\infty.$
Then in the conditions and notation of Lemma~\ref{lem1A} we may
choose the number $k_0\in {\Bbb N}$ for which there is $0<N=N(k_0,
\Vert Q\Vert_1, D^{\,\prime})<\infty,$ independent on $m$ and $f,$
such that
$$M(\Gamma_m)\leqslant N,\qquad m\geqslant M_0=M_0(k_0)\,,$$
where $\Gamma_m$ is a family of paths $\gamma:[0, 1]\rightarrow D$
in $D$ such that $f(\gamma)\in \Gamma(|\gamma_{1, m}|, |\gamma_{2,
m}|, D^{\,\prime}).$ }
\end{lemma}
\begin{proof}
Let $k_0$ be an arbitrary number for which the statement of
Lemma~\ref{lem1A} holds. By the definition of $\gamma_{1, m}$ and of
the family $\Gamma_m$ we may write that
\begin{equation}\label{eq7A}
\Gamma_m=\Gamma_m^1\cup \Gamma_m^2\,,
\end{equation}
where $\Gamma_m^1$ is a family of paths $\gamma\in\Gamma_m$ such
that $f(\gamma)\in \Gamma(|\widetilde{\alpha}|, |\gamma_{2, m}|,
D^{\,\prime})$ and $\Gamma_m^2$ is a family of paths
$\gamma\in\Gamma_m$ such that $f(\gamma)\in
\Gamma(|\widetilde{\alpha}_m|, |\gamma_{2, m}|, D^{\,\prime}).$

\medskip
Taking into account the notation of Lemma~\ref{lem1A}, we put
$$\varepsilon_0:=\min\{{\rm dist}\,(|\widetilde{\alpha}|,
\overline{g_{k_0}}), {\rm dist}\,(|\widetilde{\alpha}|,
|\widetilde{\beta}|)\}>0\,.$$
Let us consider the covering $\bigcup\limits_{x\in
|\widetilde{\alpha}|}B(x, \varepsilon_0/4)$ of
$|\widetilde{\alpha}|.$ Since $|\widetilde{\alpha}|$ is a compactum
in $D^{\,\prime},$ there are numbers $i_1,\ldots, i_{N_0}$ such that
$|\widetilde{\alpha}|\subset \bigcup\limits_{i=1}^{N_0} B(z_i,
\varepsilon_0/4),$ where $z_i\in |\widetilde{\alpha}|$ for
$1\leqslant i\leqslant N_0.$ By~\cite[Theorem~1.I.5.46]{Ku}
\begin{equation}\label{eq5E}
\Gamma(|\widetilde{\alpha}|, |\gamma_{2, m}|,
D^{\,\prime})>\bigcup\limits_{i=1}^{N_0} \Gamma(S(z_i,
\varepsilon_0/4), S(z_i, \varepsilon_0/2), A(z_i, \varepsilon_0/4,
\varepsilon_0/2))\,.
\end{equation}
Fix $\gamma\in \Gamma_m^1,$ $\gamma:[0, 1]\rightarrow D,$
$\gamma(0)\in |\widetilde{\alpha}|,$ $\gamma(1)\in |\gamma_{2, m}|.$
It follows from~(\ref{eq5E}) that $f(\gamma)$ has a subpath
$f(\gamma)_1:=f(\gamma)|_{[p_1, p_2]}$ such that
$$f(\gamma)_1\in \Gamma(S(z_i, \varepsilon_0/4), S(z_i,
\varepsilon_0/2), A(z_i, \varepsilon_0/4,  \varepsilon_0/2))$$ for
some $1\leqslant i\leqslant N_0.$ Then $\gamma|_{[p_1, p_2]}$ is a
subpath of $\gamma$ and belongs to~$\Gamma_f(z_i, \varepsilon_0/4,
\varepsilon_0/2),$ because
$$f(\gamma|_{[p_1, p_2]})=f(\gamma)|_{[p_1, p_2]}\in\Gamma(S(z_i,
\varepsilon_0/4), S(z_i, \varepsilon_0/2), A(z_i, \varepsilon_0/4,
\varepsilon_0/2)).$$ Thus
\begin{equation}\label{eq6E}
\Gamma_m^1>\bigcup\limits_{i=1}^{N_0} \Gamma_f(z_i, \varepsilon_0/4,
\varepsilon_0/2)\,.
\end{equation}
Put
$$\eta(t)= \left\{
\begin{array}{rr}
4/\varepsilon_0, & t\in [\varepsilon_0/4, \varepsilon_0/2],\\
0,  &  t\not\in [\varepsilon_0/4, \varepsilon_0/2]
\end{array}
\right. \,.$$
Observe that the function~$\eta$ satisfies the
relation~(\ref{eqA2}). Then, by the definition of $f$
in~(\ref{eq2*A}), by the relation~(\ref{eq6E}) and due to the
subadditivity of the modulus of families of paths
(see~\cite[Theorem~6.2]{Va}), we obtain that
\begin{equation}\label{eq1H}
M(\Gamma_m^1)\leqslant \sum\limits_{i=1}^{N_0} M(\Gamma_f(z_i,
\varepsilon_0/4, \varepsilon_0/2))\leqslant \sum\limits_{i=1}^{N_0}
\frac{N_04^n\Vert Q\Vert_1}{\varepsilon^n_0}\,,\qquad m\geqslant
M_0\,,
\end{equation}
where $\Vert Q\Vert_1=\int\limits_{D^{\,\prime}}Q(x)\,dm(x).$
In addition, by~\cite[Theorem~1.I.5.46]{Ku},
$\Gamma_m^2>\Gamma_f(\overline{x_0}, r_{k_0+1}, r_{k_0}).$
Arguing similarly as above, we put
$$\eta(t)= \left\{
\begin{array}{rr}
1/(r_{k_0}-r_{k_0+1}), & t\in [r_{k_0+1}, r_{k_0}],\\
0,  &  t\not\in [r_{k_0+1}, r_{k_0}]
\end{array}
\right. \,.$$
Now, by the last relation we obtain that
\begin{equation}\label{eq4DD}
M(\Gamma_m^2)\leqslant \frac{\Vert
Q\Vert_1}{(r_{k_0}-r_{k_0+1})^n}\,,\quad m\geqslant M_0\,.
\end{equation}
Thus, by~(\ref{eq7A}), (\ref{eq1H}) and~(\ref{eq4DD}), due to the
subadditivity of the modulus of families of paths
(see~\cite[Theorem~6.2]{Va}), we obtain that
$$M(\Gamma_m)\leqslant
\left(\frac{N_04^n}{\varepsilon^n_0}+\frac{1}{(r_{k_0}-r_{k_0+1})^n}\right)\Vert
Q\Vert_1\,,\quad m\geqslant M_0\,.$$
The right part of the last relation does not depend on~$m,$ so we
may put
$$N:=\left(\frac{N_04^n}{\varepsilon^n_0}+\frac{1}{(r_{k_0}-r_{k_0+1})^n}\right)\Vert
Q\Vert_1\,.$$ Lemma~\ref{lem4A} is proved.~$\Box$
\end{proof}

\medskip
Given $\delta>0, M>0$ domains $D, D^{\,\prime}\subset {\Bbb R}^n,$
$n\geqslant 2,$ and a continuum $A\subset D^{\,\prime}$ denote by
${\frak S}_{\delta, A, M}(D, D^{\,\prime})$ the family of all open
discrete and closed mappings $f$ of $D$ onto $D^{\,\prime}$ such
that the condition~(\ref{eq2*A}) holds for any $y_0\in D^{\,\prime}$
and such that $h(f^{\,-1}(A),
\partial D)\geqslant~\delta$ and $\Vert Q_f\Vert_{L^1(D^{\,\prime})}\leqslant M.$
The following theorem with certain differences in its formulation is
given in~\cite[Theorem~2]{Sev$_4$}, and its proof is completely
similar to proof of this statement. However, for completeness, we
present it in full in the text.

\medskip
\begin{theorem}\label{th2}
{\sl Let $D$ be a domain with a weakly flat boundary, and let
$D^{\,\prime}$ be a regular domain. Now, any $f\in{\frak S}_{\delta,
A, M}(D, D^{\,\prime})$ has a continuous extension
$\overline{f}:\overline{D}\rightarrow \overline{D^{\,\prime}}_P,$
wherein $\overline{f}(\overline{D})=\overline{D^{\,\prime}}_P$ and,
in addition, the family ${\frak S}_{\delta, A, M }(\overline{D},
\overline{D^{\,\prime}})$ of all extended mappings
$\overline{f}:\overline{D}\rightarrow \overline{D^{\,\prime}}_P$ is
equicontinuous in  $\overline{D}.$ }
\end{theorem}

\medskip
\begin{proof}
The possibility of continuous extension of the mapping~$f\in {\frak
S}_{\delta, A, M}(D, D^{\,\prime})$ to the boundary of $D$ follows
by Theorem~\ref{th3} and Corollary~\ref{cor2}. The equicontinuity
of~${\frak S}_{\delta, A, M}(D, D^{\,\prime})$ at inner points
of~$D$ is proved in~\cite[Theorem~1.1]{SSD}.

\medskip
Let us to show the equicontinuity of~${\frak S}_{\delta, A, М
}(\overline{D}, \overline{D^{\,\prime}})$ on $\partial D.$ Assume
the contrary. Now, there is a point $z_0\in
\partial D,$ a number~$\varepsilon_0>0,$ a sequence $z_m\in
\overline{D}$ and a mapping $\overline{f}_m\in {\frak S}_{\delta, A,
M }(\overline{D}, \overline{D^{\,\prime}})$ such that
$z_m\rightarrow z_0$ as $m\rightarrow\infty$ and
\begin{equation}\label{eq12A}
\rho(\overline{f}_m(z_m),
\overline{f}_m(z_0))\geqslant\varepsilon_0,\quad m=1,2,\ldots ,
\end{equation}
where $\rho$ is some of possible metrics
in~$\overline{D^{\,\prime}}_P$ defined in~(\ref{eq1A}). Since
$f_m=\overline{f}_m|_{D}$ has a continuous extension to
$\overline{D},$ we may assume that $z_m\in D$ and, in addition,
there is one more sequence $z^{\,\prime}_m\in D,$
$z^{\,\prime}_m\rightarrow z_0$ as $m\rightarrow\infty$ such that
$\rho(f_m(z^{\,\prime}_m), \overline{f}_m(z_0))\rightarrow 0$ as
$m\rightarrow\infty.$ In this case, it follows from~(\ref{eq12A})
that
$$
\rho(f_m(z_m), f_m(z^{\,\prime}_m))\geqslant\varepsilon_0/2,\quad
m\geqslant m_0\,.
$$
Since~$D^{\,\prime}$ is regular, the
space~$\overline{D^{\,\prime}}_P$ is compact. Thus, we may assume
that~$f_m(z_m)$ and $f_m(z_m^{\,\prime})$ converge to some $P_1,
P_2\in \overline{D^{\,\prime}}_P,$ $P_1\ne P_2,$ as
$m\rightarrow\infty.$ Let $d_m$ and $g_m$ be sequences of decreasing
domains corresponding to prime ends $P_1$ and $P_2,$ respectively.
Due to~\cite[Lemma~3.1]{IS}, cf.~\cite[Lemma~1]{KR$_2$}, we may
assume that the sequence of cuts $\sigma_m$ which corresponds
to~$d_m,$ $m=1,2,\ldots, $ belongs to spheres~$S(\overline{x_0},
r_m),$ where $\overline{x_0}\in
\partial D^{\,\prime}$ and $r_m\rightarrow 0$ as $m\rightarrow\infty.$
Put $x_0, y_0\in A$ such that $x_0\ne y_0$ and $x_0\ne P_1\ne y_0,$
where the continuum~$A\subset D^{\,\prime}$ is taken from the
conditions of Theorem~\ref{th2}. Without loss of generality, we may
assume that $d_1\cap g_1=\varnothing$ and $x_0, y_0\not\in d_1\cup
g_1.$

\medskip
By Lemmas~\ref{lem1A} and~\ref{lem4A}, we may find disjoint paths
$\gamma_{1,m}:[0, 1]\rightarrow D^{\,\prime}$ and $\gamma_{2,m}:[0,
1]\rightarrow D^{\,\prime}$ and a number $N> 0$ such that
$\gamma_{1, m}(0)=x_0,$ $\gamma_{1, m}(1)=f_m(z_m),$ $\gamma_{2,
m}(0)=y_0,$ $\gamma_{2, m}(0)=f_m(z^{\,\prime}_m),$ wherein
\begin{equation}\label{eq15}
M(\Gamma_m)\leqslant N\,,\quad m\geqslant M_0\,,
\end{equation}
where $\Gamma_m$ consists of those and only those paths $\gamma$ в
$D$ for which $f_m(\gamma)\in\Gamma(|\gamma_{1, m}|, |\gamma_{2,
m}|, D^{\,\prime})$ (see Figure~\ref{fig6}).
\begin{figure}
  \centering\includegraphics[width=300pt]{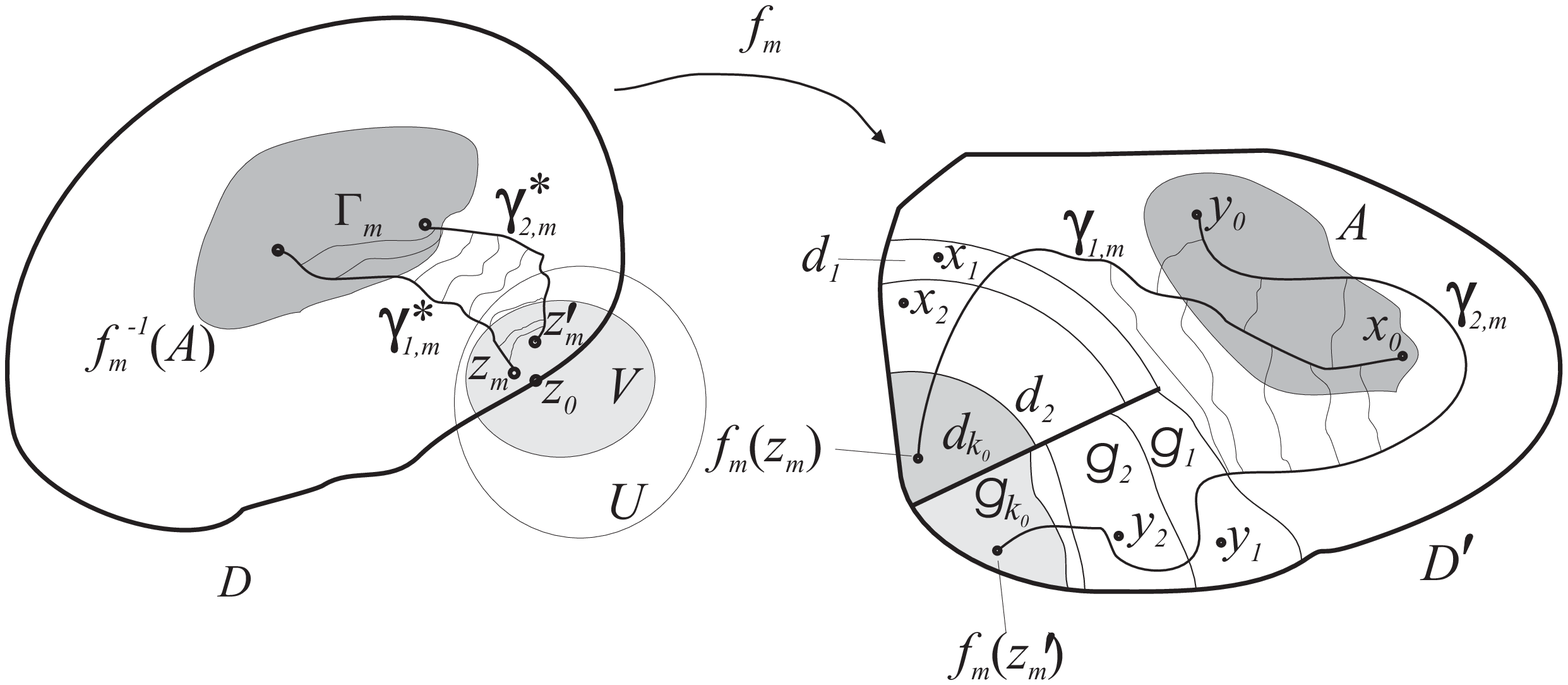}
  \caption{To proof of Theorem~\ref{th2}.}\label{fig6}
 \end{figure}
On the other hand, let~$\gamma^*_{1,m}$ and~$\gamma^*_{2,m}$ be the
total $f_m$-lifting of the paths~$\gamma_{1,m}$ and~$\gamma_{2,m}$
starting at the points~$z_m$ and $z^{\,\prime}_m,$ respectively
(such lifting exist by~\cite[Lemma~ 3.7]{Vu}). Now,
$\gamma^*_{1,m}(1)\in f^{\,-1}_m(A)$ and $\gamma^*_{2,m}(1)\in
f^{\,-1}_m(A).$ Since by the condition $h(f^{\,-1}_{m}(A),
\partial D)>\delta>0,$ $m=1,2,\ldots \,,$ we obtain that
$$h(|\gamma^*_{1, m}|)\geqslant h(z_m, \gamma^*_{1,m}(1)) \geqslant
(1/2)\cdot h(f^{\,-1}_m(A), \partial D)>\delta/2\,,$$
\begin{equation}\label{eq14A}
h(|\gamma^*_{2, m}|)\geqslant h(z^{\,\prime}_m, \gamma^*_{2,m}(1))
\geqslant (1/2)\cdot h(f^{\,-1}_m(A), \partial D)>\delta/2
\end{equation}
for sufficiently large $m\in {\Bbb N}.$
Choose the ball $U:=B_h(z_0, r_0)=\{z\in\overline{{\Bbb R}^n}: h(z,
z_0)<r_0\},$ where $r_0>0$ and $r_0<\delta/4.$ Observe that
$|\gamma^*_{1, m}|\cap U\ne\varnothing\ne |\gamma^*_{1, m}|\cap
(D\setminus U)$ for sufficiently large $m\in{\Bbb N},$ because
$h(f_m(|\gamma_{1, m}|))\geqslant \delta/2$ and $z_m\in|\gamma^*_{1,
m}|,$ $z_m\rightarrow z_0$ as $m\rightarrow\infty.$ Arguing
similarly, we may conclude that~$|\gamma^*_{2, m}|\cap
U\ne\varnothing\ne |\gamma^*_{2, m}|\cap (D\setminus U).$ Since
$|\gamma^*_{1, m}|$ and $|\gamma^*_{2, m}|$ are continua,
by~\cite[Theorem~1.I.5.46]{Ku}
\begin{equation}\label{eq8AA}
|\gamma^*_{1, m}|\cap \partial U\ne\varnothing, \quad |\gamma^*_{2,
m}|\cap
\partial U\ne\varnothing\,.
\end{equation}
Put $P:=N>0,$ where $N$ is a number from the relation~(\ref{eq15}).
Since $D$ has a weakly flat boundary, we may find a neighborhood
$V\subset U$ of $z_0$ such that
\begin{equation}\label{eq9AA}
M(\Gamma(E, F, D))>N
\end{equation}
for any continua $E, F\subset D$ with $E\cap
\partial U\ne\varnothing\ne E\cap \partial V$ and $F\cap \partial
U\ne\varnothing\ne F\cap \partial V.$ Observe that
\begin{equation}\label{eq10AA}
|\gamma^*_{1, m}|\cap \partial V\ne\varnothing, \quad |\gamma^*_{2,
m}|\cap
\partial V\ne\varnothing\,.\end{equation}
for sufficiently large $m\in {\Bbb N}.$ Indeed, $z_m\in
|\gamma^*_{1, m}|$ and $z^{\,\prime}_m\in |\gamma^*_{2, m}|,$ where
$z_m, z^{\,\prime}_m\rightarrow z_0\in V$ as $m\rightarrow\infty.$
Thus, $|\gamma^*_{1, m}|\cap V\ne\varnothing\ne |\gamma^*_{2,
m}|\cap V$ for sufficiently large $m\in {\Bbb N}.$ Besides that,
$h(V)\leqslant h(U)=2r_0<\delta/2$ and $|\gamma^*_{1, m}|\cap
(D\setminus V)\ne\varnothing$ because $h(|\gamma^*_{1,
m}|)>\delta/2$ by~(\ref{eq14A}). Then $|\gamma^*_{1, m}|\cap\partial
V\ne\varnothing$ (see~\cite[Theorem~1.I.5.46]{Ku}). Similarly,
$h(V)\leqslant h(U)=2r_0<\delta/2.$ Now, since by~(\ref{eq14A})
$h(|\gamma^*_{2, m}|)>\delta/2,$ we obtain that $|\gamma^*_{2,
m}|\cap (D\setminus V)\ne\varnothing.$
By~\cite[Theorem~1.I.5.46]{Ku}, $|\gamma^*_{1, m}|\cap\partial
V\ne\varnothing.$ Thus, (\ref{eq10AA}) is proved. By~(\ref{eq9AA}),
(\ref{eq8AA}) and (\ref{eq10AA}), we obtain that
\begin{equation}\label{eq6a}
M(\Gamma(|\gamma^*_{1, m}|, |\gamma^*_{2, m}|, D))>N\,.
\end{equation}
The inequality~(\ref{eq6a}) contradicts to~(\ref{eq15}), since
$\Gamma(|\gamma^*_{1, m}|, |\gamma^*_{2, m}|, D)\subset \Gamma_m$
and thus
$$M(\Gamma(|\gamma^*_{1, m}|, |\gamma^*_{2, m}|, D))
\leqslant M(\Gamma_m)\leqslant N\,.$$
The obtained contradiction indicates the incorrectness of the
assumption in~(\ref{eq12A}). The theorem is proved.~$\Box$
\end{proof}

\medskip
One of the versions of the following statement is established
in~\cite[item~v, Lemma~2]{SevSkv$_1$} for homeomorphisms and
''good'' boundaries, see also~\cite[Lemma~4.1]{SevSkv$_2$}. Let us
also point out the case relating to mappings with branching and bad
boundaries, see~\cite[Lemma~6.1]{SSD}, as well as the case of bad
boundaries and homeomorphisms, see~\cite[Lemma~2.13]{SSI}. We are
interested in the ''most general'' case when the mapping is only
open and discrete, and the mapped domain is regular.

\medskip
\begin{lemma}\label{lem3}
{\sl\, Let $n\geqslant 2,$ and let $D$ and $D^{\,\prime}$ be domains
in ${\Bbb R}^n.$ Assume that $D$ has a weakly flat boundary, none of
the components of which degenerates into a point, and $D^{\,\prime}$
is regular. Let $A$ be a non-degenerate continuum in $D^{\,\prime}$
and $\delta>0.$ Assume that $f_m$ is a sequence of open discrete and
closed mappings of $D$ onto $D^{\,\prime}$ satisfying the following
condition: for any $m=1,2,\ldots$ there is a continuum $A_m\subset
D,$ $m=1,2,\ldots ,$ such that $f_m(A_m)=A$ and $h(A_m)\geqslant
\delta>0.$ If there is $0<M_1<\infty$ such that $f_m$
satisfies~(\ref{eq2*A}) at any $y_0\in D^{\,\prime}$ and
$m=1,2,\ldots $ with some $Q=Q_m(y)$ for which $\Vert
Q_m\Vert_{L^1(D^{\,\prime})}\leqslant M_1,$ then there exists
$\delta_1>0$ such that
$$h(A_m,
\partial D)>\delta_1>0\quad \forall\,\, m\in {\Bbb
N}\,.$$
}
\end{lemma}

\begin{proof}
Due to the compactness of the space~$\overline{{\Bbb R}^n}$ the
boundary of the domain $D$ is not empty and is compact, so that the
distance~$\overline{{\Bbb R}^n}$$h(A_m,
\partial D)$ is well-defined.

\medskip
We will prove from the opposite. Suppose that the conclusion of the
lemma is not true. Then for each $k\in{\Bbb N}$ there is a number
$m=m_k$ such that $h(A_{m_k},
\partial D)<1/k.$ We may
assume that the sequence $m_k$ is increasing by $k.$ Since $A_{m_k}$
is compact, there are $x_k\in A_{m_k}$ і $y_k\in
\partial D$ such that $h(A_{m_k},
\partial D)=h(x_k, y_k)<1/k$ (see Figure~\ref{fig3A}).
\begin{figure}[h]
\centerline{\includegraphics[scale=0.6]{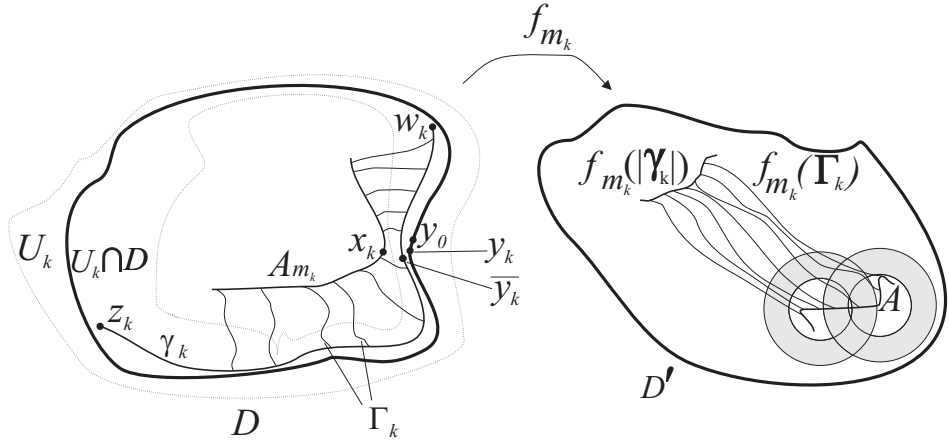}} \caption{To
proof of Lemma~\ref{lem3}}\label{fig3A}
\end{figure}
Since $\partial D$ is a compact set, we may consider that
$y_k\rightarrow y_0\in
\partial D$ as $k\rightarrow \infty.$ Now we also have that
%
$x_k\rightarrow y_0\in \partial D$ as $k\rightarrow \infty.$
%
Let $K_0$ be a component of $\partial D$ containing $y_0.$
Obviously, $K_0$ is a continuum in $\overline{{\Bbb R}^n}.$ Since
$\partial D$ is weakly flat, by Theorem~\ref{th3} $f_{m_k}$ has a
continuous extension $\overline{f}_{m_k}:\overline{D}\rightarrow
\overline{D^{\,\prime}}_P.$ Moreover, the mapping
$\overline{f}_{m_k}$ is uniformly continuous in $\overline{D}$ for
any fixed $k,$ because $\overline{f}_{m_k}$ is continuous on the
compact set $\overline{D}.$ Now, for any $\varepsilon>0$ there is
$\delta_k=\delta_k(\varepsilon)<1/k$ such that
\begin{equation}\label{eq3B}
\rho(\overline{f}_{m_k}(x),\overline{f}_{m_k}(x_0))<\varepsilon
\end{equation}
$$
\forall\,\, x,x_0\in \overline{D},\quad h(x, x_0)<\delta_k\,, \quad
\delta_k<1/k\,,$$
where $\rho$ is one of the possible metric
in~$\overline{D^{\,\prime}}_P,$
\begin{equation}\label{eq1AA}
\rho(x, y):=|g^{\,-1}(x)-g^{\,-1}(y)|\,,
\end{equation}
where the element $g^{\,-1}(x)$ is understood as some (single) point
of the boundary $D_0$ for $x\in E_{D^{\,\prime}},$ which is
well-defined due to~\cite[Theorem~2.1]{IS}; see
also~\cite[теорема~4.1]{Na$_2$}.
Let $\varepsilon>0$ be some number such that
\begin{equation}\label{eq5D}
\varepsilon<(1/2)\cdot {\rm dist}\,(\partial D_0, g^{\,-1}(A))\,,
\end{equation}
where $A$ is a continuum from the conditions of lemma and
$g:D_0\rightarrow D^{\,\prime}$ is a quasiconformal mapping of $D_0$
onto $D,$ while $D$ is a domain with a quasiconformal boundary
corresponding to the definition of the metric $\rho.$
Put $B_h(x_0, r)=\{x\in \overline{{\Bbb R}^n}: h(x, x_0)<r\}.$ Given
$k\in {\Bbb N},$ we set
$$B_k:=\bigcup\limits_{x_0\in K_0}B_h(x_0, \delta_k)\,,\quad k\in {\Bbb
N}\,.$$
Since $B_k$ is a neighborhood of a continuum $K_0,$
by~\cite[Lemma~2.2]{HK} there is a neighborhood $U_k$ of $K_0$ such
that $U_k\subset B_k$ and $U_k\cap D$ is connected. We may consider
that $U_k$ is open, so that $U_k\cap D$ is linearly path connected
(see~\cite[Proposition~13.1]{MRSY}). Let $h(K_0)=m_0.$ Then we may
find $z_0, w_0\in K_0$ such that $h(K_0)=h(z_0, w_0)=m_0.$ Thus,
there are sequences $\overline{y_k}\in U_k\cap D,$ $z_k\in U_k\cap
D$ and $w_k\in U_k\cap D$ such that $z_k\rightarrow z_0,$
$\overline{y_k}\rightarrow y_0$ and $w_k\rightarrow w_0$ as
$k\rightarrow\infty.$ We may consider that
\begin{equation}\label{eq2B}
h(z_k, w_k)>m_0/2\quad \forall\,\, k\in {\Bbb N}\,.
\end{equation}
Since the set $U_k\cap D$ is linearly connected, we may joint the
points $z_k,$ $\overline{y_k}$ and $w_k$ using some path
$\gamma_k\in U_k\cap D.$ As usually, we denote by $|\gamma_k|$ the
locus of the path $\gamma_k$ in $D.$ Then $f_{m_k}(|\gamma_k|)$ is a
compact set in $D^{\,\prime}.$ If $x\in|\gamma_k|,$ then we may find
$x_0\in K_0$ such that $x\in B(x_0, \delta_k).$ Fix  $\omega\in
A\subset D.$ Since $x\in|\gamma_k|$ and, in addition, $x$ is an
inner point of $D,$ we may use the notation $f_{m_k}(x)$ instead
$\overline{f}_{m_k}(x).$ By~(\ref{eq3B}) and~(\ref{eq5D}), and by
the triangle inequality, we obtain that
$$\rho(f_{m_k}(x),\omega)\geqslant
\rho(\omega,
\overline{f}_{m_k}(x_0))-\rho(\overline{f}_{m_k}(x_0),f_{m_k}(x))\geqslant$$
\begin{equation}\label{eq4C}
\geqslant {\rm dist}\,(\partial D_0, g^{\,-1}(A))-(1/2)\cdot {\rm
dist}\,(\partial D_0, g^{\,-1}(A))=(1/2)\cdot {\rm dist}\,(\partial
D_0, g^{\,-1}(A))>\varepsilon
\end{equation}
for sufficiently large $k\in {\Bbb N}.$ Passing to $\inf$
in~(\ref{eq4C}) over all $x\in |\gamma_k|$ and $\omega\in A,$ we
obtain that
\begin{equation}\label{eq18}
\rho(f_{m_k}(|\gamma_k|), A)>\varepsilon, \quad k=1,2,\ldots .
\end{equation}
We now show that there is $\varepsilon_1>0$ such that
\begin{equation}\label{eq6B}
{\rm dist}\,(f_{m_k}(|\gamma_k|), A)>\varepsilon_1, \quad\forall\,\,
k=1,2,\ldots \,.
\end{equation}
Indeed, let~(\ref{eq6B}) be violated. Then for the number
$\varepsilon_l=1/l,$ $l=1,2,\ldots$ there are $\xi_l\in
|\gamma_{k_l}|$ and $\zeta_l\in A$ such that
\begin{equation}\label{eq21}
|f_{m_{k_l}}(\xi_l)-\zeta_l|<1/l\,,\quad l=1,2,\ldots \,.
\end{equation}
Without loss of generality, we may assume that the sequence of
numbers $k_l,$ $l=1,2,\ldots,$ is increasing. Since $A$ is a compact
set, we may consider that the sequence $\zeta_l$ converge to some
$\zeta_0\in A$ as $l\rightarrow\infty.$ By the triangle inequality
and from~(\ref{eq21}) it follows that
\begin{equation}\label{eq22}
|f_{m_{k_l}}(\xi_l)-\zeta_0|\rightarrow 0\,,\quad
l\rightarrow\infty\,.
\end{equation}
On the other hand, recall that $\rho(f_{m_k}(x),
\omega)=|g^{\,-1}(f_{m_k}(x))-g^{\,-1}(\omega)|,$ where
$g:D_0\rightarrow D^{\,\prime}$  is a quasiconformal mapping of $
D_0$ onto $D^{\,\prime}$ (see~(\ref{eq1AA})). In particular, $g^{\,-
1}$ is a continuous mapping in $D^{\,\prime}.$ Thus, by the triangle
inequality and~(\ref{eq22}) we obtain that
$$|g^{\,-1}(f_{m_{k_l}}(\xi_l))-g^{\,-1}(\zeta_l)|\leqslant$$
\begin{equation}\label{eq23}
\leqslant
|g^{\,-1}(f_{m_{k_l}}(\xi_l))-g^{\,-1}(\zeta_0)|+|g^{\,-1}(\zeta_0)-g^{\,-1}(\zeta_l)|\rightarrow
0,\quad l\rightarrow\infty\,.\end{equation}
However, by the definition of $\rho$ and by~(\ref{eq23}) we obtain
that
$$\rho(f_{m_{k_l}}(|\gamma_{k_l}|), A)\leqslant \rho(f_{m_{k_l}}(\xi_l), \zeta_l)=
|g^{\,-1}(f_{m_{k_l}}(\xi_l))-g^{\,-1}(\zeta_l)|\rightarrow 0, \quad
l\rightarrow\infty\,,$$
that contradicts to~(\ref{eq18}). The resulting contradiction
indicates the validity of relation~(\ref{eq6B}).

\medskip
We cover the set $A$ with balls $B(x, \varepsilon/4),$ $x\in A.$
Since $A$ is compact, we may assume that $A\subset
\bigcup\limits_{i=1}^{M_0}B(x_i, \varepsilon/4),$ $x_i\in A,$
$i=1,2,\ldots, M_0,$ $1\leqslant M_0<\infty.$ By the definition,
$M_0$ depends only on $A,$ in particular, $M_0$ does not depend on
$k.$ Put
$$
\Gamma_k:=\Gamma(A_{m_k}, |\gamma_k|, D)\,.
$$
Let $\Gamma_{ki}:=\Gamma_{f_{m_k}}(x_i, \varepsilon/4,
\varepsilon/2),$ in other words, $\Gamma_{ki}$ consists of all paths
$\gamma:[0, 1]\rightarrow D$ such that $f_{m_k}(\gamma(0))\in S(x_i,
\varepsilon/4),$ $f_{m_k}(\gamma(1))\in S(x_i, \varepsilon/2)$ і
$\gamma(t)\in A(x_i, \varepsilon/4, \varepsilon/2)$ for $0<t<1.$ We
show that
\begin{equation}\label{eq6C}
\Gamma_k>\bigcup\limits_{i=1}^{M_0}\Gamma_{ki}\,.
\end{equation}
Indeed, let $\widetilde{\gamma}\in \Gamma_k,$ in other words,
$\widetilde{\gamma}:[0, 1]\rightarrow D,$ $\widetilde{\gamma}(0)\in
A_{m_k},$ $\widetilde{\gamma}(1)\in |\gamma_k|$ and
$\widetilde{\gamma}(t)\in D$ for $0\leqslant t\leqslant 1.$ Then
$\gamma^{\,*}:=f_{m_k}(\widetilde{\gamma})\in \Gamma(A,
f_{m_k}(|\gamma_k|), D^{\,\prime}).$ Since the balls $B(x_i,
\varepsilon/4),$ $1\leqslant i\leqslant M_0,$  form the coverage of
the compact set $A,$ we may find $i\in {\Bbb N}$ such that
$\gamma^{\,*}(0)\in B(x_i, \varepsilon/4)$ and $\gamma^{\,*}(1)\in
f_{m_k}(|\gamma_k|).$ By the relation~(\ref{eq6B}),
$|\gamma^{\,*}|\cap B(x_i, \varepsilon/4)\ne\varnothing\ne
|\gamma^{\,*}|\cap (D^{\,\prime}\setminus B(x_i, \varepsilon/4)).$
Thus, by~\cite[Theorem~1.I.5.46]{Ku} there is $0<t_1<1$ such that
$\gamma^{\,*}(t_1)\in S(x_i, \varepsilon/4).$ We may assume that
$\gamma^{\,*}(t)\not\in B(x_i, \varepsilon/4)$ for $t>t_1.$ Set
$\gamma_1:=\gamma^{\,*}|_{[t_1, 1]}.$ By~(\ref{eq6B}) it follows
that $|\gamma_1|\cap B(x_i, \varepsilon/2)\ne\varnothing\ne
|\gamma_1|\cap (D\setminus B(x_i, \varepsilon/2)).$ Thus,
by~\cite[Theorem~1.I.5.46]{Ku} there is $t_1<t_2<1$ such that
$\gamma^{\,*}(t_2)\in S(x_i, \varepsilon/2).$ We may assume that
$\gamma^{\,*}(t)\in B(x_i, \varepsilon/2)$ for any $t<t_2.$ Putting
$\gamma_2:=\gamma^{\,*}|_{[t_1, t_2]},$ we observe that a path
$\gamma_2$ is a subpath of $\gamma^{\,*},$ which belongs to
$\Gamma(S(x_i, \varepsilon/4), S(x_i, \varepsilon/2), A(x_i,
\varepsilon/4, \varepsilon/2)).$

Finally, $\widetilde{\gamma}$ has a subpath
$\widetilde{\gamma_2}:=\widetilde{\gamma}|_{[t_1, t_2]}$ such that
$f_{m_k}\circ\widetilde{\gamma_2}=\gamma_2,$ while
$$\gamma_2\in \Gamma(S(x_i, \varepsilon/4), S(x_i, \varepsilon/2),
A(x_i, \varepsilon/4, \varepsilon/2))\,.$$ Thus, the
relation~(\ref{eq6C}) is proved. Set
$$\eta(t)= \left\{
\begin{array}{rr}
4/\varepsilon, & t\in [\varepsilon/4, \varepsilon/2],\\
0,  &  t\not\in [\varepsilon/4, \varepsilon/2]\,.
\end{array}
\right. $$
Observe that $\eta$ satisfies the relation~(\ref{eqA2}) for
$r_1=\varepsilon/4$ and $r_2=\varepsilon/2.$ Since $f_{m_k}$
satisfies the relation~(\ref{eq2*A}), we obtain that
\begin{equation}\label{eq8C}
M(\Gamma_{f_{m_k}}(x_i, \varepsilon/4, \varepsilon/2))\leqslant
(4/\varepsilon)^n\cdot\Vert Q\Vert_1<M_0<\infty\,.
\end{equation}
By~(\ref{eq6C}) and (\ref{eq8C}) and due to the subadditivity of the
modulus of families of paths, we obtain that
\begin{equation}\label{eq4B}
M(\Gamma_k)\leqslant
\frac{4^nM_0}{\varepsilon^n}\int\limits_{D^{\,\prime}}Q(y)\,dm(y)\leqslant
M_1\cdot M_0<\infty\,.
\end{equation}
Arguing similarly to the proof of relations~(\ref{eq14A}) and using
the condition~(\ref{eq2B}), we obtain that
$M(\Gamma_k)\rightarrow\infty$ as $k\rightarrow\infty,$ which
contradicts to~(\ref{eq4B}). The resulting contradiction proves the
lemma.
\end{proof}

\medskip
Given domains $D, D^{\,\prime}\subset {\Bbb R}^n,$ points $a\in D,$
$b\in D^{\,\prime}$ and a number $M_0>0$ denote by ${\frak S}_{a, b,
M_0}(D, D^{\,\prime})$ the family of open discrete and closed
mappings $f$ of $D$ onto $D^{\,\prime}$ satisfying the
relation~(\ref{eq2*A}) for some $Q=Q_f,$ $\Vert
Q\Vert_{L^1(D^{\,\prime})}\leqslant M_0$ for any $y_0\in f(D),$ such
that $f(a)=b.$ The following statement was proved
in~\cite[Theorem~7.1]{SSD} in the case of a fixed function~$Q$ (cf.
Theorem~4.2 in~\cite{SD}).

\medskip
\begin{theorem}\label{th4}
{\sl Assume that $D$ has a weakly flat boundary, none of the
components of which degenerates into a point, and $D^{\,\prime}$ is
regular. Then any $f\in {\frak S}_{a, b, M_0}(D, D^{\,\prime})$ has
a continuous extension $\overline{f}:\overline{D}\rightarrow
\overline{D^{\,\prime}}_P,$ while
$\overline{f}(\overline{D})=\overline{D^{\,\prime}}_P$ and, in
addition, the family ${\frak S}_{a, b, M_0}(\overline{D},
\overline{D^{\,\prime}})$ of all extended mappings
$\overline{f}:\overline{D}\rightarrow \overline{D^{\,\prime}}_P$ is
equicontinuous in $\overline{D}.$ }
\end{theorem}

\medskip
\begin{proof} The possibility of continuous extension of
$f\in {\frak S}_{a, b, M_0}(D, D^{\,\prime})$ to a continuous
mapping $\overline{f}:\overline{D}\rightarrow
\overline{D^{\,\prime}}_P$ is a statement of Theorem~\ref{th3}, as
well as the equality
$\overline{f}(\overline{D})=\overline{D^{\,\prime}}_P.$ The
equicontinuity of ${\frak S}_{a, b, M_0}(D, D^{\,\prime})$ at inner
points of $D$ is a result of~\cite[Theorem~1.1]{SSD}.

\medskip
It remains to establish the equicontinuity of the family of extended
mappings $\overline{f}:\overline{D}\rightarrow
\overline{D^{\,\prime}}_P$ at the boundary points of the domain~$D.$

\medskip
We prove this statement from the opposite. Assume that the family
${\frak S}_{a, b, M_0}(\overline{D}, \overline{D^{\,\prime}})$ is
not equicontinuous at some point $x_0\in\partial D.$ Then there are
points $x_m\in D$ and mappings $f_m\in {\frak S}_{a, b,
M_0}(\overline{D}, \overline{D^{\,\prime}}),$ $m=1,2,\ldots ,$ such
that $x_m\rightarrow x_0$ as $m\rightarrow\infty,$ moreover,
\begin{equation}\label{eq15B}
\rho(f_m(x_m), f_m(x_0))\geqslant\varepsilon_0\,,\quad
m=1,2,\ldots\,.
\end{equation}
for some $\varepsilon_0>0,$ where $\rho$ is one of the metrics in
$\overline{D^{\,\prime}}_P$ (see, e.g., \cite[Remark~2]{KR$_1$}). We
choose in an arbitrary way the point $y_0\in D^{\,\prime},$ $y_0\ne
b,$ and join it to the point $b$ by some path in $D^{\,\prime},$
which we denote by $\alpha.$ Let $A:=|\alpha|$ and let $A_m$ be a
total $f_m$-lifting of $\alpha$ starting at $a$ (it exists
by~\cite[Lemma~3.7]{Vu}). Observe that $h(A_m,
\partial D)>0$ due to the closeness of~$f_m.$ Now, the following two cases are possible:
either $h(A_m)\rightarrow 0$ as $m\rightarrow\infty,$ or
$h(A_{m_k})\geqslant\delta_0>0$ as $k\rightarrow\infty$ for some
increasing sequence of numbers $m_k$ and some $\delta_0>0.$

\medskip
In the first of these cases, obviously, $h(A_m, \partial D)\geqslant
\delta>0$ for some $\delta>0.$ Then, by Theorem~\ref{th2}, the
family $\{f_m\}_{m=1}^{\infty}$ is equicontinuous at the point
$x_0,$ however, this contradicts the condition~(\ref{eq15B}).

In the second case, if $h(f(A_{m_k}))\geqslant\delta_0>0$ for
sufficiently large $k,$ we also have that $f(A_{m_k}, \partial
D)\geqslant \delta_1>0$ for some $\delta_1> 0$ by Lemma~\ref{lem3}.
Again, by Theorem~\ref{th2}, the family $\{f_{m_k}\}_{k=1}^{\infty}
$ is equicontinuous at the point $x_0,$ and this contradicts the
condition~(\ref{eq15B}).

Thus, in both of the two possible cases, we came to a
contradiction~(\ref{eq15B}), and this indicates the incorrect
assumption of the absence of the equicontinuity of the family
${\frak S}_{a, b, M_0}(D, D^{\,\prime})$ in $\overline{D}.$ The
theorem is proved.~$\Box$
\end{proof}

\section{Compactness of families of solutions of the Dirichlet
problem}

{\it Proof of Theorem~\ref{th2A}.} In general, we will use the
scheme of proving Theorem~1.2 in~\cite{SD}.

\medskip
{\textbf I.} Let $f_m\in \frak{F}^{\mathcal{M}}_{\varphi, \Phi,
z_0}(D),$ $m=1,2,\ldots .$ By Stoilow's factorization theorem (see,
e.g., \cite[5(III).V]{St}) a mapping $f_m$ has a representation
\begin{equation}\label{eq2E}
f_m=\varphi_m\circ g_m\,,
\end{equation}
where $g_m$ is some homeomorphism, and $\varphi_m$ is some analytic
function. By Lemma~1 in~\cite{Sev$_2$}, the mapping $g_m$ belongs to
the Sobolev class $W_{\rm loc}^{1, 1}(D)$ and has a finite
distortion. Moreover, by~\cite[(1).C, Ch.~I]{A}
\begin{equation}\label{eq1B}
{f_m}_z={\varphi_m}_z(g_m(z)){g_m}_z,\qquad
{f_m}_{{\overline{z}}}={\varphi_m}_z(g_m(z)){g_m}_{\overline{z}}
\end{equation}
for almost all $z\in D.$ Therefore, by the relation~(\ref{eq1B}),
$J(z, g_m)\ne 0$ for almost all $z\in D,$ in addition,
$K_{\mu_{f_m}}(z)=K_{\mu_{g_m}}(z).$

\medskip
\textbf{II.} We prove that $\partial g_m (D)$ contains at least two
points. Suppose the contrary. Then either $g_m(D)={\Bbb C},$ or
$g_m(D)={\Bbb C}\setminus \{a\},$ where $a\in {\Bbb C}.$ Consider
first the case $g_m(D)={\Bbb C}.$ By Picard's theorem
$\varphi_m(g_m(D))$ is the whole plane, except perhaps one point
$\omega_0\in {\Bbb C}.$ On the other hand, for every $m=1,2,\ldots$
the function $u_m(z):={\rm Re}\,f_m(z)={\rm
Re}\,(\varphi_m(g_m(z)))$ is continuous on the compact set
$\overline{D}$ under the condition~(\ref{eq1A}) by the continuity
of~$\varphi.$ Therefore, there exists $C_m>0$ such that $|{\rm
Re}\,f_m(z)|\leqslant C_m$ for any $z\in D,$ but this contradicts
the fact that $\varphi_m(g_m(D))$ contains all points of the complex
plane except, perhaps, one. The situation $g_m(D)={\Bbb C}\setminus
\{a\},$ $a\in {\Bbb C},$ is also impossible, since the domain
$g_m(D)$ must be simply connected in ${\Bbb C}$ as a homeomorphic
image of the simply connected domain $D.$

\medskip
Therefore, the boundary of the domain $g_m(D)$ contains at least two
points. Then, according to Riemann's mapping theorem, we may
transform the domain $g_m(D)$ onto the unit disk ${\Bbb D}$ using
the conformal mapping $\psi_m.$ Let $z_0\in D $ be a point from the
condition of the theorem. By using an auxiliary conformal mapping
$$\widetilde{\psi_m}(z)=\frac{z-(\psi_m\circ
g_m)(z_0)}{1-z\overline{(\psi_m\circ g_m)(z_0)}}$$ of the unit disk
onto itself we may consider that $(\psi_m\circ g_m)(z_0)=0.$ Now,
by~(\ref{eq2E}) we obtain that
$$
f_m=\varphi_m\circ g_m= \varphi_m\circ\psi^{\,-1}_m\circ\psi_m\circ
g_m=F_m\circ G_m\,,\quad m=1,2,\ldots\,,
$$
where $F_m:=\varphi_m\circ\psi^{\,-1}_m,$ $F_m:{\Bbb D}\rightarrow
{\Bbb C},$ and $G_m=\psi_m\circ g_m.$
Obviously, a function $F_m$ is analytic, and $G_m$ is a regular
Sobolev homeomorphism in $D.$ In particular, ${\rm Im}\,F_m(0)=0$
for any $m\in {\Bbb N}.$

\medskip
\textbf{III.} We prove that the $L^1$-norms of the functions
$K_{\mu_{G_m}}(z)$ are bounded from above by some universal positive
constant $C> 0$ over all $m=1,2,\ldots .$ Indeed, by the convexity
of the function $\Phi$ in~(\ref{eq1D}) and by~\cite[Proposition~5,
I.4.3]{Bou},  the slope $\left[\Phi(t)-\Phi(0)\right]/t$ is a
non-decreasing function. Hence there exist constants $t_0>0$ and
$C_1>0$ such that
\begin{equation}\label{eq6D}
\Phi(t)\geqslant C_1\cdot t\qquad \forall\,\, t\in [t_0, \infty)\,.
\end{equation}
Fix $m\in {\Bbb N}\,.$ By~(\ref{eq1D}) and~(\ref{eq6D}), we obtain
that
$$\int\limits_D K_{\mu_{G_m}}(z)\,dm(z)=
\int\limits_{\{z\in D: K_{\mu_{G_m}}(z)<t_0\}}
K_{\mu_{G_m}}(z)\,dm(z)+\int\limits_{\{z\in D:
K_{\mu_{G_m}}(z)\geqslant t_0\}} K_{\mu_{G_m}}(z)\,dm(z)\leqslant$$
$$\leqslant t_0\cdot m(D)+\frac{1}{C_1}\int\limits_D
\Phi(K_{\mu_{G_m}}(z))\,dm(z)\leqslant$$$$\leqslant t_0\cdot
m(D)+\frac{\sup\limits_{z\in D}(1+|z|^2)^2}{C_1}\int\limits_D
\Phi(K_{\mu_{G_m}}(z))\cdot\frac{1}{(1+|z|^2)^2}\,dm(z)\leqslant$$
$$
\leqslant t_0\cdot m(D)+\frac{\sup\limits_{z\in
D}(1+|z|^2)^2}{C_1}\mathcal{M}(D)<\infty\,,
$$
because $\mathcal{M}(D)<\infty$ by the assumption of the theorem.

\medskip
\textbf{IV.} We prove that each map $G_m,$ $m=1,2,\ldots ,$ has a
continuous extension to $E_D,$ in addition, the family of extended
maps $\overline{G}_m,$ $m=1,2,\ldots ,$ is equicontinuous in
$\overline{D}_P.$ Indeed, as proved in item~\textbf{III},
$K_{\mu_{G_m}}\in L^1(D).$ By~\cite[Theorem~3]{KPRS} (see
also~\cite[Theorem~3.1]{LSS}) each $G_m,$ $m=1,2,\ldots, $ is a ring
$Q$-homeomorphism in $\overline{D}$ for $Q=K_{\mu_{G_m}}(z),$ where
$\mu$ is defined in~(\ref{eq2C}), and $K_{\mu}$ my be calculated by
the formula~(\ref{eq1}). Note that the unit disk ${\Bbb D}$ is a
uniform domain as a finitely connected flat domain at its boundary
with a finite number of boundary components (see, for example,
~\cite[Theorem~6.2 and Corollary~6.8]{Na$_1$}). Then it is desirable
the conclusion is a statement of Theorem~\ref{th1}.

\medskip
\textbf{V.} Let us prove that the inverse homeomorphisms
$G^{\,-1}_m,$ $m=1,2,\ldots ,$ have a continuous extension
$\overline{G}^{\,-1}_m$ to $\partial {\Bbb D}$ in terms of prime
ends in $D,$ and $\{\overline{G}_m^{\,-1}\}_{m=1}^{\infty}$ is
equicontinuous in $\overline{\Bbb D}$ as a family of mappings from
$\overline{\Bbb D}$ to $\overline{D}_P.$ Since by the
item~\textbf{IV} mappings $G_m,$ $m=1,2,\ldots, $ are ring
$K_{\mu_{G_m}}(z)$-homeomorphisms in $D,$ the corresponding inverse
mappings $G^{\,-1}_m$ satisfy~(\ref{eq2*A}) (in this case, $D$
corresponds the unit disk ${\Bbb D}$ in (\ref{eqA2}),  $f\mapsto
G_m,$ $Q\mapsto K_{\mu_{G_m}}(z),$ and $f(D)\mapsto D$). Since
$G^{\,-1}_m(0)=z_0$ for any $m=1,2,\ldots ,$ the possibility of a
continuous extension of $G^{\,-1}_m$ to $\partial {\Bbb D},$ and the
equicontinuity of ${\{\overline{G}_m^{\,-1}\}}_{m=1}^{\infty}$ in
terms of $G_m^{\,-1}:\overline{\Bbb D}\rightarrow \overline{D}_P$
follows by Theorem~\ref{th4}.

\medskip
\textbf{VI.} Since, as proved above the family
${\{G_m\}}^{\infty}_{m=1}$ is equicontinuous in~$D,$ according to
Arzela-Ascoli criterion there exists an increasing subsequence of
numbers $m_k,$ $k=1,2,\ldots ,$ such that $G_{m_k}$ converges
locally uniformly in $D$ to some continuous mapping $G:D\rightarrow
\overline{{\Bbb C}}$ as $k\rightarrow\infty$  (see, e.g.,
\cite[Theorem~20.4]{Va}). By~\cite[Lemma~2.1]{SD}, either $G$ is a
homeomorphism with values in ${\Bbb R}^n,$ or a constant
in~$\overline{{\Bbb R}^n}.$ Let us prove that the second case is
impossible. Let us apply the approach used in proof of the second
part of Theorem~21.9 in~\cite{Va}. Suppose the contrary: let
$G_{m_k}(x)\rightarrow c=const$ as $k\rightarrow\infty.$ Since
$G_{m_k}(z_0)=0$ for all $k=1,2,\ldots ,$ we have that $c=0.$ By
item~\textbf{V}, the family of mappings~$G^{\,-1}_m,$ $m=1,2,\ldots
,$ is equicontinuous in ${\Bbb D}.$ Then
$$h(z, G^{\,-1}_{m_k}(0))=h(G^{\,-1}_{m_k}(G_{m_k}(z)), G^{\,-1}_{m_k}(0))\rightarrow 0$$
as $k\rightarrow\infty,$ which is impossible because $z$ is an
arbitrary point of the domain $D.$ The obtained contradiction
refutes the assumption made above. Thus, $G:D\rightarrow {\Bbb C}$
is a homeomorphism.

\medskip
\textbf{VII.} According to~\textbf{V}, the family of mappings
$\{\overline{G}_m^{\,-1}\}_{m=1}^{\infty}$ is equicontinuous
in~$\overline{\Bbb D}.$ By the Arzela-Ascoli criterion (see, e.g.,
\cite[Theorem~20.4]{Va}) we may consider that
$\overline{G}^{\,-1}_{m_k}(y),$ $k=1,2,\ldots, $ converges to some
mapping $\widetilde{F}:\overline{{\Bbb D}}\rightarrow \overline{D}$
as $k\rightarrow\infty$ uniformly in~$\overline{D}.$ Let us to prove
that $\widetilde{F}=\overline{G}^{\,-1}.$ For this purpose, we show
that~$G(D)={\Bbb D}.$ Fix $y\in {\Bbb D}.$ Since $G_{m_k}(D)={\Bbb
D}$ for every $k=1,2,\ldots, $ we obtain that $G_{m_k}(x_k)=y$ for
some $x_k\in D.$ Since $D$ is regular, the metric space
$(\overline{D}_P, \rho)$ is compact. Thus, we may assume that
$\rho(x_k, x_0)\rightarrow 0$ as $k\rightarrow\infty,$ where
$x_0\in\overline{D}_P.$ By the triangle inequality and the
equicontinuity of ${\{\overline{G}_m\}}^{\infty}_{m=1}$ onto
$\overline{D}_P,$ see~\textbf{IV}, we obtain that
$$|\overline{G}
(x_0)-y|=|\overline{G}(x_0)-\overline{G}_{m_k}(x_k)|\leqslant
|\overline{G}(x_0)-\overline{G}_{m_k}(x_0)|+|\overline{G}_{m_k}(x_0)-\overline{G}_{m_k}(x_k)|\rightarrow
0$$
as $k\rightarrow\infty.$ Thus, $\overline{G}(x_0)=y.$ Observe that
$x_0\in D,$ because $G$ is a homeomorphism. Since $y\in {\Bbb D}$ is
arbitrary, the equality $G(D)={\Bbb D}$ is proved. In this case,
$G^{\,-1}_{m_k}\rightarrow G^{\,-1}$ locally uniformly in ${\Bbb D}$
as $k\rightarrow\infty$ (see, e.g., \cite[Lemma~3.1]{RSS}). Thus,
$\widetilde{F}(y)=G^{\,-1}(y)$ for every $y\in {\Bbb D}.$

Finally, since $\widetilde{F}(y)=G^{\,-1}(y)$ for any $y\in {\Bbb
D}$ and, in addition, $\widetilde{F}$ has a continuous extension on
$\partial{\Bbb D},$ due to the uniqueness of the limit at the
boundary points we obtain that
$\widetilde{F}(y)=\overline{G}^{\,-1}(y)$ for $y\in \overline{{\Bbb
D}}.$ Therefore, we have proved
that~$\overline{G}^{\,-1}_{m_k}\rightarrow \overline{G}^{\,-1}$
uniformly in~$\overline{\Bbb D}$ with as $k\rightarrow\infty$ with
respect to the metrics $\rho$ in $\overline{D}_P.$

\medskip
\textbf{VIII.} By~\textbf{VII,} for $y=e^{i\theta}\in
\partial {\Bbb D}$
\begin{equation}\label{eq4E}
{\rm
Re\,}F_{m_k}(e^{i\theta})=\varphi\left(\overline{G}^{\,-1}_{m_k}(e^{i\theta})\right)\rightarrow
\varphi\left(\overline{G}^{\,-1}(e^{i\theta})\right)
\end{equation}
as $k\rightarrow\infty$ uniformly on $\theta\in [0, 2\pi).$ Since by
the construction ${\rm Im\,}F_{m_k}(0)=0$ for any $k=1,2,\ldots,$ by
the Schwartz formula (see, e.g., \cite[section~8.III.3]{GK}) the
analytic function $F_{m_k}$ is uniquely restored by its real part,
namely,
\begin{equation}\label{eq4A}
F_{m_k}(y)=\frac{1}{2\pi i}\int\limits_{S(0,
1)}\varphi\left(\overline{G}^{\,-1}_{m_k}(t)\right)\frac{t+y}{t-y}\cdot\frac{dt}{t}\,.
\end{equation}
Set
\begin{equation}\label{eq5B}
F(y):=\frac{1}{2\pi i}\int\limits_{S(0,
1)}\varphi\left(\overline{G}^{\,-1}(t)\right)\frac{t+y}{t-y}\cdot\frac{dt}{t}\,.
\end{equation}
Let $K\subset {\Bbb D}$ be an arbitrary compact set, and let $y\in
K.$ By~(\ref{eq4A}) and~(\ref{eq5A}) we obtain that
\begin{equation}\label{eq11A}
|F_{m_k}(y)-F(y)|\leqslant \frac{1}{2\pi}\int\limits_{S(0,
1)}\bigl|\varphi(\overline{G}^{\,-1}_{m_k}(t))-\varphi(\overline{G}
^{\,-1}(t))\bigr|\left|\frac{t+y}{t-y}\right|\,|dt|\,.
\end{equation}
Since $K$ is compact, there is $0<R_0=R_0(K)<\infty$ such that
$K\subset B(0, R_0).$ By the triangle inequality $|t+y|\leqslant
1+R_0$ and $|t-y|\geqslant |t|-|y|\geqslant 1-R_0$ for $y\in K$ and
any $t\in {\Bbb S}^1.$ Thus
$$
\left|\frac{t+y}{t-y}\right|\leqslant \frac{1+R_0}{1-R_0}:=M=M(K)\,.
$$
Put $\varepsilon>0.$ By~(\ref{eq4E}), for a number
$\varepsilon^{\,\prime}:=\frac{\varepsilon}{M}$ there is
$N=N(\varepsilon, K)\in {\Bbb N}$ such that
$\bigl|\varphi\left(\overline{G}^{\,-1}_{m_k}(t)\right)-\varphi\left
(\overline{G}^{\,-1}(t)\right)\bigr|<\varepsilon^{\,\prime}$ for any
$k\geqslant N(\varepsilon)$ and $t\in {\Bbb S}^1.$ Now,
by~(\ref{eq11A})
\begin{equation}\label{eq13A}
|F_{m_k}(y)-F(y)|<\varepsilon \quad \forall\,\,k\geqslant N\,.
\end{equation}
It follows from~(\ref{eq13A}) that the sequence $F_{m_k}$ converges
to $F$ as $k\rightarrow\infty$ in the unit disk locally uniformly.
In particular, we obtain that ${\rm Im\,}F(0)=0.$ Note that $F$ is
analytic function in ${\Bbb D}$ (see remarks made at the end of
item~8.III in~\cite{GK}), and

$${\rm Re}\,F(re^{i\psi})=\frac{1}{2\pi}\int\limits_0^{2\pi}
\varphi\left(\overline{G}^{\,-1}(e^{i\theta})\right)\frac{1-r^2}{1-2r\cos(\theta-\psi)+r^2}\,d\theta$$
for $z=re^{i\psi}.$
By~\cite[Theorem~2.10.III.3]{GK}
\begin{equation}\label{eq15A}
\lim\limits_{\zeta\rightarrow z}{\rm
Re}\,F(\zeta)=\varphi(\overline{G}^{\,-1}(z))\quad\forall\,\,z\in\partial
{\Bbb D}\,.
\end{equation}
Observe that $F$ either is a constant or open and discrete (see,
e.g., \cite[Ch.~V,  I.6 and II.5]{St}). Thus, $f_{m_k}=F_{m_k}\circ
G_{m_k}$ converges to $f=F\circ G$ locally uniformly as
$k\rightarrow\infty,$ where $f=F\circ G$ either is a constant or
open and discrete. Moreover, by~(\ref{eq15A})
$$\lim\limits_{\zeta\rightarrow P}{\rm Re\,}
f(\zeta)= \lim\limits_{\zeta\rightarrow P}{\rm
Re\,}F(G(\zeta))=\varphi(G^{\,-1}(G(P)))=\varphi(P)\,.$$

\textbf{IX.} Since by~\textbf{VI} $G$ is a homeomorphism,
by~\cite[Lemma~1 and Theorem~1]{L} $G$  is a regular solution of the
equation~(\ref{eq2C}) for some function~$\mu:{\Bbb C}\rightarrow
{\Bbb D}.$ Since the set of points of the function $F,$ where its
Jacobian is zero, consist only of isolated points (see~\cite[Ch.~V,
5.II and 6.II]{St}), $f$ is regular solution of the Dirichlet
problem~(\ref{eq2C})--(\ref{eq1A}) whenever $F\not\equiv const.$
Note that the relation~(\ref{eq1D}) holds for the corresponding
function $K_{\mu}=K_{\mu_f}$ (see e.g.~\cite[Lemma~1]{L}).
Therefore, $f\in\frak{F}^{\mathcal{M}}_{\varphi, \Phi,
z_0}(D).$~$\Box$


\medskip
{\bf \noindent Oleksandr Dovhopiatyi} \\
{\bf 1.} Zhytomyr Ivan Franko State University,  \\
40 Bol'shaya Berdichevskaya Str., 10 008  Zhytomyr, UKRAINE \\
alexdov1111111@gmail.com

\medskip
{\bf \noindent Evgeny Sevost'yanov} \\
{\bf 1.} Zhytomyr Ivan Franko State University,  \\
40 Bol'shaya Berdichevskaya Str., 10 008  Zhytomyr, UKRAINE \\
{\bf 2.} Institute of Applied Mathematics and Mechanics\\
of NAS of Ukraine, \\
1 Dobrovol'skogo Str., 84 100 Slavyansk,  UKRAINE\\
esevostyanov2009@gmail.com

\end{document}